\documentclass[12pt]{amsart}  
\usepackage{amsmath,graphicx,amssymb,amsthm,amssymb,enumerate}
\usepackage{url,caption}

\usepackage{tikz}
\usetikzlibrary{calc,arrows,trees,positioning,matrix}

\usepackage{graphicx}
\usepackage{lscape}
\usepackage{indentfirst}
\usepackage{latexsym}
\usepackage{multirow}
\usepackage{tabls}
\usepackage{wrapfig}
\usepackage{slashbox}
\usepackage{longtable}
\newcommand{\isomto}{\overset{\sim}{\rightarrow}}
\newcommand{\Z}{\ensuremath{\mathbf Z}} 
 
\newcommand{\R}{\ensuremath{\mathbf R}} 
\newcommand{\C}{\ensuremath{\mathbf C}} 
\newcommand{\captionme}[1]{\captionof{figure}[#1]{#1}}

\newcommand{\isom}{\cong}

\newcommand{\cross}{\ensuremath{^\times}}

\tikzstyle myBG=[line width=.6mm,opacity=1.0]
\swapnumbers
\theoremstyle{plain}
\newtheorem{theorem}{Theorem}[section]
\newtheorem{lemma}[theorem]{Lemma}
\newtheorem{corollary}[theorem]{Corollary}

\theoremstyle{definition}

\theoremstyle{remark}

\DeclareMathOperator{\siz}{size}
\DeclareMathOperator{\care}{char}
\DeclareMathOperator{\size}{size}
\DeclareMathOperator{\inv}{inv}
\DeclareMathOperator{\adj}{Ad}
\DeclareMathOperator{\lie}{Lie}

\DeclareMathOperator{\orb}{O_\gamma}
\DeclareMathOperator{\twisted}{TO_\delta}

\DeclareMathOperator{\meas}{meas}

\newcommand{\wae}{\ensuremath{\widetilde{W}_a}}
\newcommand{\wa}{\ensuremath{W_a}}

\newcommand{\build}{\ensuremath{\mathcal{B}}}
\newcommand{\sigd}{\ensuremath{^\sigma_\delta}}
\newcommand{\ints}{\ensuremath{\mathcal{O}}}
\DeclareMathOperator{\val}{val}

\newcommand{\intsunits}{\ensuremath{\ints\cross}}
\newcommand{\sapt}{\ensuremath{\mathfrak{A}}}
\newcommand{\apt}{\ensuremath{\mathfrak{A}_T}}

\newcommand{\zhi}{\ensuremath{Z(\mathcal{H})}}
\newcommand{\zhie}{\ensuremath{Z(\mathcal{H}_{E})}}
\newcommand{\zhil}{\ensuremath{Z(\mathcal{H}_{L}})}

\DeclareMathOperator{\gal}{Gal}

\newcommand{\delsig}{\ensuremath{\delta\sigma}}
\newcommand{\vd}{\ensuremath{\val\det}}

\newcommand{\xwgz}{\ensuremath{X_w^0(\gamma)}}
\newcommand{\xwdz}{\ensuremath{X_{w}^0(\delta\sigma)}}

\newcommand{\primes}{\ensuremath{\mathfrak{p}}}

\allowdisplaybreaks[3]

\begin{document}
\newcommand{\plainh}{\mathcal{H}}
\title{Base Change for the Iwahori-Hecke Algebra of $GL_2$}
\author{Walter Ray-Dulany}
\address{University of Maryland, College Park, USA}
\email{walter.rd.math@gmail.com}
\begin{abstract}
Let $F$ be a non-Archimedean local field of characteristic not equal to $2$, let $E/F$ be a finite unramified extension field, and let $\sigma$ be a generator of $\gal(E,F)$. Let $f$ be an element of $Z(\plainh_{I_E})$, the center of the Iwahori-Hecke algebra for $GL_2(E)$, and let $b$ be the Iwahori base change homomorphism from $Z(\plainh_{I_E})$ to $Z(\plainh_{I_F})$, the center of the Iwahori-Hecke algebra for $GL_2(F)$ \cite{haineslemma}. This paper proves the matching of the $\sigma$-twisted orbital integral over $GL_2(E)$ of $f$ with the orbital integral over $GL_2(F)$ of $bf$.
To do so, we compute the orbital and $\sigma$-twisted orbital integrals of the Bernstein functions $z_\mu$. These integrals are computed by relating them to counting problems on the set of edges in the building for $SL_2$. Surprisingly, the integrals are found to be somewhat independent of the conjugacy class over which one is integrating. The matching of the integrals follows from direct comparison of the results of these computations. The fundamental lemma proved here is an important ingredient in the study of Shimura varieties with Iwahori level structure at a prime $p$ \cite{hainesshimura}.
\end{abstract}
\maketitle
%
%
%
%
%
%
%
%
\setlength{\parskip}{0em}
\setcounter{tocdepth}{1}
\tableofcontents
\small\normalsize

\pagenumbering{arabic}


\section{Introduction and Statement of Main Results}\label{intro}
The base change fundamental lemma arises in the study of base change for automorphic representations. Based on earlier work by Saito \cite{saito}, it was first introduced by Shintani \cite{shintani} and Langlands \cite{langlands} as the main step in the comparison of the trace formula for $GL_2(L)$, $L$ a number field, with the twisted trace formula for $GL_2({L'})$, ${L'}$ a cyclic extension of $L$ of degree $l$. Let $F$ be the completion of $L$ at a finite prime $\mathfrak{p}$, and let $E$ be an unramified cyclic extension of $F$. Among other terms, the trace formula involves orbital integrals on \emph{spherical Hecke algebras} $\mathcal{H}_K$ of compactly supported locally constant $\C$-valued functions on $GL_2(F)$ which are $K$-bi-invariant, $K$ a hyperspecial maximal compact subgroup of $GL_2(F)$, while the twisted trace formula involves twisted orbital integrals of functions in $\mathcal{H}_{K_E}$, $K_E$ a hyperspecial maximal compact subgroup of $GL_2(E)$.

Langlands showed the following.
\begin{theorem}
\cite[5.10]{langlands} 
Let $F$ be a non-Archimedean local field of characteristic zero
\footnote{
Although Langlands assumes $\care(F)=0$, his proof works without alteration in all characteristics except $2$.
},
 let $E$ be an unramified extension of $F$ of prime degree $l$, let $\sigma$ generate $\gal(E,F)$, let $\phi\in \mathcal{H}_{K_{E}}$, and let $\gamma\in GL_2(F)$. If $\gamma=N(\delta)$ for some $\delta\in GL_2(E)$, then
$$\int_{G_\delta(F)\backslash GL_2(E)} \phi(g^{-1}\delta\sigma(g))dg=\xi(\gamma)\int_{T(F)\backslash GL_2(F)} b(\phi)(g^{-1}\gamma g)dg.$$ Here $\xi(\gamma)$ is $1$ unless $\gamma$ is central and $\delta\neq g^{-1}\delta'\sigma(g)$ for any central element $\delta'$ and $g\in GL_2(E)$, in which case it is $-1$. If  $\gamma\neq N(\delta)$ for any $\delta\in GL_2(E)$, then
$$\int_{T(F)\backslash GL_2(F)} b(\phi)(g^{-1}\gamma g)dg=0.$$
\end{theorem}
Here $b$ is the \emph{base change homomorphism for spherical hecke algebras}, an operation which is the analog for functions of local base change for unramified representations. Further, $N$ is the norm map $N(\delta)=\delta\sigma(\delta)\cdots\sigma^{l-1}(\delta)$, $G_\delta(F)=\{g\in GL_2(E): g^{-1}\delta\sigma(g)=\delta\}$ is the \emph{twisted centralizer} of $\delta$ in $GL_2(E)$, $T(F)$ is the centralizer of $\gamma$ in $GL_2(F)$, and $dg$ represents compatibly defined measures on the respective quotients.  The integral over $T(F)\backslash GL_2(F)$ is called the \emph{orbital integral} of $b(\phi)$ and is denoted $\orb(b(\phi))$, and the integral over $G_\delta(F)\backslash GL_2(E)$ is the \emph{twisted orbital integral} of $\phi$, denoted $\twisted(\phi)$. Langlands' proof relies upon considering the actions of $\gamma$ and $\delta\sigma$ upon the Bruhat-Tits building of $SL_2$, and reduces in part to a vertex-counting problem. The base change homomorphism is in this case somewhat complicated, and computing $b(\phi)$ requires passage to $\C[X_*(S)]^W$ using the Satake isomorphism. Here $S$ is a split maximal torus in $GL_2(F)$, $X_*(S)$ is its group of cocharacters and $W$ is the finite Weyl group of $GL_2$.

This theorem and its analogs for other unramified connected reductive groups have come to be known as the \emph{base change fundamental lemma for spherical Hecke algebras}. In an approach motivated by that of Langlands, Kottwitz \cite{kotthesis} proved it  for $GL_3$, 
and later \cite{kotunit} provided a key tool for future progress, proving the theorem for the special case that $\phi$ is the unit element of $\mathcal{H}_{E}$.
\footnote{In formulating the theorem for general groups, one can no longer work with single orbital or twisted orbital integrals, and must work instead with stable orbital and twisted orbital integrals. Our work will not require the stable version, and so we do not go into this formulation of the theorem. See \cite{kotunit}.}
 The methods used in these papers do not work for other groups, and further progress relied on \cite{kotunit} and the trace formula, which is only available for characteristic zero.
 The next major work was that of Arthur and Clozel \cite{arthurclozel}, who proved the theorem for $GL_n$. The theorem was solved in general by Clozel \cite{clozel} and Labesse \cite{labesse} for any unramified connected reductive algebraic group. 

Interest in the study of the base change fundamental lemma for the centers of Iwahori-Hecke algebras comes out the study of Shimura varieties with Iwahori level structure at a prime $p$. In this case, the functions of interest belong to the center of the Iwahori-Hecke algebra $Z(\mathcal{H}_I)$, where $I$ is an Iwahori subgroup of the group $G$ under consideration. The base change operator $b:Z(\mathcal{H}_{I_E})\rightarrow Z(\mathcal{H}_I)$ in this case is defined using the Bernstein isomorphism rather than the Satake isomorphism (see \ref{basechangedef} for the definition of $b$). The fundamental lemma in this context was proven by Haines \cite{haineslemma} for all unramified connected reductive algebraic groups over fields of characteristic zero, again using trace formula methods.

In this paper we prove the following base change fundamental lemma for the centers of Iwahori-Hecke algebras of $GL_2$.
\begin{theorem}
Let $F$ be a non-Archimedean local field of characteristic not equal to $2$, let $E$ be a finite unramified extension of $F$, let $\phi\in Z(\mathcal{H}_{I_E})$, and let $\gamma\in G=GL_2(F)$ be semi-simple regular. If there exists a $\delta \in G_E=GL_2(E)$ such that $N(\delta)$ is conjugate to $\gamma$, then
$$\twisted(\phi) = \orb(b(\phi)).$$
If $\gamma$ is not conjugate to the norm of any $\delta\in G_E$, then 
$$\orb(b(\phi))=0.$$
\end{theorem}
This result is new for function fields. It is sufficient to prove the theorem for a basis of $Z(\mathcal{H}_{I_E})$. A convenient basis is the one consisting of the Bernstein functions $z_\mu$, $\mu\in X_*(S)$ dominant (see \ref{bernsec} to recall the  definition and properties of the $z_\mu$). We prove the following explicit computation of the twisted orbital integrals of the $z_\mu$, using methods adapted from \cite{langlands}.
\begin{theorem}\label{mainresult}
Let $F$ be non-Archimedean local field of characteristic not equal to $2$, and let $E$ be an unramified extension of $F$ of degree $f$. Let $\delta\in G_E$ be such that $N(\delta)$ is conjugate to a semi-simple regular elliptic $\gamma\in G$, let the eigenvalues of $\gamma$ be $\lambda_1$ and $\lambda_2$ and let $T$ be the centralizer of $\gamma$ in $G$. Normalize $dg_E$ so that $I_E$ is given measure $1$ and normalize $dg_{T}$ so that the measure of the maximal compact open subgroup of $T$ is $1$. Let $\ints$ be the ring of integers of $F$ and let $q$ be the cardinality of the residue field of $F$.
\begin{enumerate}
\item Assume $T$ splits over an unramified quadratic extension of $F$. Define $a$ by $\frac{1}{2}\val_F\left(\frac{(\lambda_1-\lambda_2)^2}{\lambda_1 \lambda_2}\right)=a$. Then 
$$\twisted(z_{\mu})=\begin{cases} 
			0, & \val\det(\delta)\neq \size(\mu)\\
			2(q^a-1)\left(\frac{q+1}{q-1}\right), & \ell(\mu)=0\\
			2q^{-f\ell(\mu)/2}(1-(-q)^{f\ell(\mu)}), & \ell(\mu)\neq 0.\end{cases}
$$ 
\item Assume $T$ splits over a ramified quadratic extension $F'$ of $F$. Let $d_T$ be the distance in the building $\build(G_{F})$ of $GL_2$ over $F'$ from the apartment associated to $T$ to the $F$-points of $\build(G_{F'})$.
 Define $a$ by $\frac{1}{2}\val_F\left(\frac{(\lambda_1-\lambda_2)^2}{\lambda_1 \lambda_2}\right)=a+\frac{d_T+1}{2}$. Then
$$\twisted(z_{\mu})=\begin{cases}
			0, & \val\det(\delta)\neq \size(\mu)\\
			\frac{2q^{a+1}-q-1}{q-1}, & \ell(\mu)=0\\
			q^{-f\ell(\mu)/2}(1-q^{f\ell(\mu)}), & \ell(\mu)\neq 0.\end{cases}$$
\end{enumerate}
\end{theorem}

One obtains the orbital integrals from these formulas as the special case $E=F$, $f=1$. Since $b(z_\mu)=z_{f\mu}$ \cite{haineslemma}, the fundamental lemma follows as a corollary of theorem \ref{mainresult}.
We remark that the formulae for the integrals in theorem \ref{mainresult} seem to be much simpler than those obtained by Langlands in \cite{langlands} for characteristic functions of the double coset $K\mu(\pi)K$, $\pi$ a uniformizer of $\ints$.\footnote{We recall Langlands' results in section \ref{oldlanglands} for the convenience of the reader.}

In the course of the research for this paper, the author found it useful to begin working with the orbital case, and only afterwards progressing to the twisted case. This paper is organized the same way; thus some of section \ref{buildchapt} is redundant, and much of chapter \ref{orbitalchapt} is only a special case of the results in chapter \ref{twistedchapt}. It is the author's belief that the practice and experience gained by separately looking at the special case of $E=F$ makes up for the redundancy.


\newcommand{\fbar}{\ensuremath{\overline{F}}}
\newcommand{\fprime}{\ensuremath{F[\sqrt{D}]}}

\section{Notation}\label{notation}

\subsection{Fields}
We denote by $F$ a non-Archimedean local field of any characteristic not equal to $2$, with ring of integers \ints\ and prime ideal $\primes\subset\ints$. The unit elements of $\ints$ will be denoted $\intsunits$. Fix a uniformizer $\pi\in\primes$. The cardinality of the residue field $\ints/\primes$ will be denoted by $q$. 

Fix an algebraic closure $\fbar$ of $F$. For each finite algebraic extension $L\subset \fbar$ of $F$, we denote the ring of integers of $L$, its prime ideal, and a uniformizer thereof by $\ints_{L}$, $\primes_{L}$, and $\pi_L$, respectively. 
 We denote the valuation on $L$ by $\val_L$ (we will often drop the subscript if $L=F$), and we normalize the valuation so that $\val_L(\pi_L)=1$. We define $|x|_L=(q^n)^{-\val_L(x)}$, where $q^n$ is the cardinality of the residue field of $L$.

We will use $E$ to denote a finite unramified algebraic extension of $F$ of degree $[E:F]=f$; there is a unique such extension in \fbar. Let $\sigma_{E/F}$ be a generator of the cyclic group $\gal(E,F)$. When no confusion may arise as to which fields are meant, we shall use simply $\sigma$.

\subsection{Groups}
Let $\mathbf{G}=GL_2$ and $\mathbf{H}=SL_2$. We will denote $GL_2(F)$ by $G$, and $GL_2(R)$ by $G_R$ for any other ring $R$ (e.g. $G_L$, $G_{\ints_L}$).
 We shall follow this style of notation for every algebraic group we use.
In $\mathbf{G}$ we fix the Borel subgroup $\mathbf{B}$ of upper triangular matrices and its maximal split torus $\mathbf{S}$ of matrices whose only non-zero entries lie along the diagonal.  
We will denote the center of $G_L$ by $Z(G_L)$.
 When we need to take the intersubsection of one of these groups with $\mathbf{H}$ we shall denote it by appending $H$ as a superscript, e.g. $S^H_L=(S\cap H)(L)$. 

The cocharacter groups of $S_L$ and $S^H_L$ are denoted $X_*(S_L)$ and $X_*(S_L^H)$, respectively. The first group is isomorphic to $\Z^2$ while the second is isomorphic to $\Z$.  We choose representatives in $S_L$ for $X_*(S_L)$ by sending $\mu=(m,n)\in X_*(S_L)$ to  
$$t_\mu=\begin{pmatrix} \pi_L^{-m} & 0 \\ 0 & \pi_L^{-n}\end{pmatrix}.$$ This identification gives an isomorphism $X_*(S_L)\isom S_L/S_{\ints_L}$.
We may embed $X_*(S_L^H)$ in $X_*(S_L)$  by sending $m\in \Z$ to $(m,-m)$ in $\Z^2$; this yields $X_*(S_L^H)\isom S_L^H/S_{\ints_L}^H$.

We will always denote by $\gamma$ a semi-simple regular elliptic element of $G$, and by $\delta$ an element of $G_E$. We will denote by
$\mathbf{T}=\mathbf{G_\gamma}$ the elliptic torus that is the centralizer of $\gamma$ and by $\mathbf{G_\delta}$ the \emph{twisted centralizer} of $\delta$, whose $F$-points are $\{g\in G_E:g^{-1}\delta\sigma(g)=\delta\}$. We call $g^{-1}\delsig(g)$ the \emph{twisted conjugate} or \emph{$\sigma$-conjugate} of $\delta$ by $g$. 

We denote by  $N:G_E\rightarrow G_E$ the \emph{norm map}, defined by $$N(\delta)=\delta\sigma_E(\delta)\sigma_E^2(\delta)\ldots\sigma_E^{f-1}(\delta).$$ 

\subsection{Buildings}\label{buildnote}
The Bruhat-Tits building of $H_L$ is denoted $\build(H_L)$ 
(see subsection \ref{slbuildsec} for details). 
The apartment of $\build(H_L)$ associated to $\mathbf{S^H}$ will be denoted $\sapt(L)$. Recall (e.g. \cite[p. 5]{haineslust}) that the Bruhat-Tits building of $G_L$ is equal to $\build(H_L)\times \R$, and its apartment $\sapt'(L)$ associated to $S_L$ is equal to $\sapt(L)\times (\Z\otimes \R)=\sapt(L)\times \R$. Following and expanding upon the terminology of \cite{haineslust} we shall mean by an \emph{extended edge} or \emph{extended vertex} of the Bruhat-Tits building of $G_L$ an object of the form $(e,c)$ or $(v,c)$, respectively, in $\build(H_L)\times \Z\subset \build(H_L)\times\R$. Denote by $\build(G_L)$ the \emph{extended building} $\build(H_L)\times\Z$ of $G_L$. 
For any operator $\alpha$ on $\build(H_L)$ or $\build(G_L)$, denote by $\build(H_L)^\alpha$ and $\build(G_L)^\alpha$ the fixed-point sets of $\alpha$ in each building.

 Suppose $\mathbf{T}$ splits over $L$. The apartment and extended apartment over $L$ associated to $\mathbf{T^H}$ and $\mathbf{T}$ will be denoted $\apt(L)$ and $\apt'(L)$, respectively (in general, a prime ($'$) indicates an object in the extended building). 
We choose a base edge $e_0$ and a base extended edge $e_0'$ within $\sapt(L)$ and $\sapt'(L)$, respectively, as follows. Let $\Lambda_0$ be the $\ints_L$-lattice $\ints_L\oplus \ints_L$ and let $\Lambda_1$ be the $\ints_L$-lattice $\pi^{-1}\ints_L\oplus\ints_L$, and let $[\Lambda]$ denote the homothety class of an $\ints_L$-lattice $\Lambda$ in $L^2$. Then $e_0=([\Lambda_0]_L,[\Lambda_1]_L)$ and $e_0'=(\Lambda_0,\Lambda_1)$. We also choose a base vertex $v_0=[\Lambda_0]_L$ and a base extended vertex $v_0'=\Lambda_0$. We think of $v_0$ as the origin of $\sapt$, $v_1=[\Lambda_1]$ as the vertex lying at $1$, and $e_0$ as the edge that is the interval $[0,1]$. 
\begin{center}
\begin{tikzpicture}
[axis/.style={<->,>=stealth'},every circle node/.style={draw,circle,
			inner sep=0mm},scale=.5,grow cyclic,
	level 1/.style={level distance=8mm,sibling angle=180},
	level 2/.style={level distance=6mm,sibling angle=68},
	level 3/.style={level distance=6mm,sibling angle=57}]
	\draw[white,myBG] (-12,0) -- (12,0);
  	\draw[black] (-12,0) -- (12,0);
  	\draw (-12,0) node[left] {$\ldots$}; 
  	\draw (12,0) node[right] {$\ldots$};
  	\node at (14,0) {$\sapt_L$};  			
   	\foreach \colorA/\colorB/\shift in {black/white/0,white/black/2.5}
  	  \foreach \pos in {-8.75-\shift,-6.25-\shift,-1.25-\shift,3.75-\shift,8.75-\shift,11.25}{
	    \node[fill=white] at (\pos,0) [circle, minimum size=1.8mm] {};
	    \node[fill=\colorA] at (\pos,0) [circle, minimum size=1.7mm] {}; 
	}
	\node[fill=white] at (-11.25,0) [circle, minimum size=1.8mm] {};
	    \node[fill=black] at (-11.25,0) [circle, minimum size=1.7mm] {};
	    
	  \node at (-1.25,.2) [above] {\scriptsize$v_0$};
	  \node at (-1.25,-.2) [below] {$0$};
	  \node at (1.25,.2) [above] {\scriptsize$v_1$};
	  \node at (1.25,-.2) [below] {$1$};
	  \node at (0,0) [above] {\scriptsize$e_0$};
\end{tikzpicture}
\captionme{The apartment $\sapt(L)$}\label{standardapt}
\end{center}

The group $G$ acts transitively on $\build(G_L)$. For any extended edge $e'$ of $\build(G_L)$, there is a $g\in G_L$ such that $e'=ge_0'$. We define the \emph{size} of $e$ to be $\val_L\det(g)$; size is the $\Z$ component of an extended edge.
We think of $\sapt'(L)$ as an infinite stack of copies of $\sapt(L)$, indexed by size. 
Under these conventions, the group $X_*(S_L)$ acts on $\sapt(L)$ by translations, with $t_{(m,n)}$ being a translation to the right by $m-n$, and on $\sapt'(L)$ by $t_{(m,n)}(e,c)=(e+(m-n),c-(m+n))$ for any edge $e$ in $\sapt(L)$.

\subsection{The Weyl, Affine Weyl, and Extended Affine Weyl Groups}\label{weylnote}
Let $N_{S_L}$ denote the normalizer in $G_L$ of $S_L$ and $N_{S_L^H}$ denote the normalizer in $H_L$ of $S_L^H$. 
The \emph{finite Weyl group} of $H_L$ is $N_{S_L^H}/S^H_L$ and is denoted $W$. It is isomorphic to $\Z/2\Z$; let $s_1$ denote the non-trivial element. Then $W$ comes with an action on $\sapt(L)$, with $s_1$ acting as reflection through $v_0$.
The finite Weyl group of $G_L$ is $W=N_{S_L}/S_L$, and is also isomorphic to $\Z/2\Z$; if we also denote its non-trivial element by $s_1$, then the action of $s_1$ on $\sapt'(L)$ is as a reflection of each copy $\sapt(L)\times\{c\}$ about $(v_0,c)$.

The two finite Weyl groups can be identified by choosing representatives of each $s_1$ in $H_L$, and both groups can be identified with  $N_{S^H_{\ints_L}}/S^H_{\ints_L}$ if we also choose those representatives to lie in the $\ints_L$-points of $H_L$. To those ends we choose  
$$s_1=\begin{pmatrix}0 & -1\\1 & 0 \end{pmatrix}$$
to represent $s_1$ in $H_{\ints_L}$. It will be clear from context whether we are using $s_1$ as an element of $W$ or of $H_{\ints_L}$.

The \emph{affine Weyl group} of $H_L$ is $\wa = N_{S^H_L}/S^H_{\ints_L}$. The groups $W=N_{S^H_{\ints_L}}/S^H_{\ints_L}$ and $X_*(S_L^H)\isom S_L^H/S^H_{\ints_L}$ may be viewed as subgroups, and then $\wa\isom X_*(S^H_L)\rtimes W$. The affine Weyl group acts simply transitively upon the set of edges in $\sapt(L)$. This action is determined by the actions of $W$ and $X_*(S_L^H)$ on $\sapt(L)$; in particular, $t_{(n,-n)}$ acts as translation to the right by $2n$ and $s_1$ acts as reflection about $0$.
 Any element of $\wa$ can be written $t_{(-1,1)}^m s_1^b$ for some $m\in \Z$ and $b\in \Z/2\Z$; we will denote this element by $(m,b)$. It is important to be aware that \emph{the action of $(m,0)$ on $\sapt(L)$ is translation to the left by 2m}.  Our earlier choices of representatives for $X_*(S^H_L)$ and $W$ give us representatives for any $(m,b)\in \wa$.  

There is another useful presentation of $\wa$. If we denote by $s_0$ the element $(-1,1)=t_{(1,-1)}s_1$, then $s_0$ acts on $\sapt(L)$ as a reflection about $v_1$. It is then the case that $\wa$ is a Coxeter group generated by the two \emph{affine reflections} $s_0$ and $s_1$ (indeed, one could define $\wa$ as the Coxeter group generated by the reflections through the two vertices defined by our choice of $e_0$). 

The \emph{extended affine Weyl group} of $G_L$ is $\wae=N_{S_L}/S_{\ints_L}$ (note: \wae\ is not a Coxeter group). As with \wa,  $X_*(S_L)$ and $W$ may be identified with subgroups of $\wae$, and then $\wae\isom X_*(S_L)\rtimes W$. The extended affine Weyl group acts simply transitively upon the set of extended edges in $\sapt'(L)$. The inclusion $X_*(S_L^H)\hookrightarrow X_*(S_L)$ permits an inclusion $\wa\hookrightarrow \wae$, and the image of $\wa$ is normal in $\wae$. If we denote by $\Omega\isom \Z$ the group $\wae/\wa$, then we have the isomorphism 
\begin{align*}
\wae &\isom \wa \rtimes \Omega\\
&\isom (X_*(S_L^H)\rtimes W)\rtimes \Omega.
\end{align*}
We choose the generator $\tau=t_{(0,-1)}s_1$ of $\Omega$; then the action of $\tau^s$ on $\sapt'(L)$ sends $(e,c)$ to $(s_1^se,c+s)$.  We may write any $w\in\wae$ as a triple $(m,b,s)=t_{(-1,1)}^m s_1^b\tau^s$ in the set (not group) $\Z\times\Z/2\Z\times \Z$. 

If we choose
$$\tau=\begin{pmatrix} 0 & 1\\ \pi & 0\end{pmatrix}$$
to be a representative of $\tau$ in $G_L$, then we have chosen for every $(m,b,s)\in \wae$ a representative in $G_L$. 
If $w=(m,b,s)$, then the \emph{size of w} is $s$. We will abuse notation by writing $w\in\wae$ for both the element of \wae\ and its representative in $G_L$. 

\subsection{The Iwahori Subgroup}\label{iwahorinote}
The group $H_L$ acts transitively on the set of edges and also on the set of vertices; similarly, $G_L$ acts transitively on the sets of extended edges and vertices. 
The standard Iwahori subgroups $I^H_L$ and $I_L$ are the subgroups of $H_L$ and $G_L$ that fix $e_0$ and $e_0'$, respectively. Denote by $K_L$ the subgroup that fixes $v_0'$. In terms of matrices, $K_L = G_{\ints_L}$ and 
$$I_L =G_L\cap \begin{pmatrix}\intsunits_L & \ints_L\\ \pi_L\ints_L & \intsunits_L\end{pmatrix}.$$
  By the construction of $I_L^H$ and $I_L$, the set of edges can be identified with $X_L=H_L/I^H_L$ and the set of extended edges  with $X'_L=G_L/I_L$ by the maps $gI^H_L\mapsto ge_0$ and $gI_L\mapsto ge_0$.

Having chosen representatives in $G_L$ for $\wae$ and in $H_L$ for $\wa$, we have the Bruhat-Tits decompositions
\begin{align*}
H_L &= \coprod_{w\in \wa} I_L^H w I_L^H\\
G_L &= \coprod_{w\in \wae} I_L w I_L,
\end{align*}
as well as bijections from $\wa$ to set of edges in $\sapt(L)$ and from $\wae$ to the set of extended edges in $\sapt'(L)$, the bijections being $w\mapsto we_0$ and $w\mapsto we'_0$, respectively. By construction, if $w\in\wae$ has size $s$ then so does its corresponding extended edge.

\subsection{Bruhat Order and Length}\label{lengthnote}
The two affine reflections $s_0$ and $s_1$ determined by our choice of $e_0$ as base edge in $\build(H_L)$ determine a length function $\ell$ from $\wa$ to the non-negative integers. 
\begin{lemma}\label{lengthlemma} The length of $w=(m,b)\in\wa$ is given by 
$$\ell(m,b)=|2m+b|,$$
where in the formula we are using $b=0$ or $1$.
\end{lemma}
\begin{proof}
Denote by $R^+=\{e_1-e_2\}$ the positive root and $R^-=\{e_2-e_1\}$ the negative root determined by our choice of $e_0$ as our base edge.
 Let $\mu=(-m,m)\in X_*(S_L^H)$ and $w\in W$. 
 The Iwahori-Matsumoto formula (see \cite{iwmat} or \cite[p. 24]{hainestest}) tells us that 
$$\ell(t_\mu w) = \sum_{\alpha\in R^+\cap wR^-} |\langle\alpha, \mu\rangle -1|+\sum_{\alpha\in R^+\cap w R^+} |\langle\alpha,\mu\rangle|.$$
When $w=1$, the we get 
$\ell((m,0)) = \ell(t_\mu) = |\langle e_1-e_2,(-m,m)\rangle| = |2m|.$
When $w=s_1$, we get 
$\ell(m,1) = \ell(t_\mu s_1) = |\langle e_1-e_2,(-m,m)\rangle -1| = |-2m-1| = |2m+1|.$
\end{proof}
For any $w\in\wa$, $\ell(w)$ is also equal to the length of the minimal gallery in $\sapt(L)$ from $e_0$ to $we_0$. For any $e_1$ and $e_2\in\wa$, define $\ell_L(e_1,e_2)$ to be the length of the minimal gallery between $e_1$ and $e_2$ in $\build(H_L)$; we shall omit the subscript when no ambiguity will arise.

We denote the Bruhat order on $\wa$ by $\leq$, and both the length and the Bruhat order extend to \wae\ by setting $\ell(m,b,s)=\ell(m,b)$ and $(m,b,s)\leq (m',b',s')$ if and only if $s=s'$ and $(m,b)\leq (m',b')$. 

For any $s\in \Z$ and $n$ a non-negative integer, the above shows there are only two $w\in \wae$ with size $s$ and length $n$. We can determine them explicitly: if $n=2r+1$ is odd, then $w=(r,1,s)$ or $(-r-1,1,s)$; if $n=2r$ is even, then $w=(\pm r,0,s)$. If $w\in\wae$ has length not equal to zero, we denote by $\overline{w}$ the element of $\wae$ distinct from $w$ with length and size equal to those of $w$; when $w$ has length zero, $\overline{w}=w$. 

Let $i\neq j\in\{0,1\}$ and $w\in\wae$. Because $\tau$ normalizes $I_L$ and $\tau^{-1}s_i\tau=s_j$ (viewed as elements of $\wae$), the existence of a decomposition of $\wa$ into a product of affine reflections immediately implies that $\overline{w}=\tau^{-1}w\tau=\tau w\tau^{-1}$.

\subsection{Relative Distance}\label{reldisnote}
We define the \emph{relative distance} function  $\inv:X'_L\times X'_L\rightarrow \wae$ as follows. Let $e_1'=g_1e_0'$ and $e_2'=g_2e_0'$ be extended edges. Then $I_Lg_1^{-1}g_2I_L$ determines a unique $w\in\wae$ by the Bruhat-Tits decomposition. Let $\inv(e_1',e_2')=w$. 
There is also a relative distance for $X_L\times X_L$. Also denoted $\inv$, it is defined  by $\inv(e_1,e_2)=w\in\wa$, where $e_i=g_ie_0$, $g_i\in H_L$, $i=1,2$, and $I^H_Lg_1^{-1}g_2I_L^H=I_L^HwI_L^H$. Then $\ell(e_1,e_2)=\ell(\inv(e_1,e_2))$.

Let $\alpha\in G_L\times \gal(L,F)$. Then $\alpha$ acts as a simplicial isometry on $\build(G_L)$. For any $w\in\wae$ define $$X_w(\alpha) =\{e'\in X_L : \inv(e',\alpha e')=w\}.$$ For any $n\in\Z$, define 
$$X^n_w(\alpha)=\{e'\in X_w(\alpha): \siz(e')=n\}.$$  
\begin{lemma}\label{sizelemma}
For any $\alpha\in G_L\times\gal(L,F)$,  any $w\in \wae$, and any $n\in \Z$, $\#X_w^n(\alpha)=\#X_{\overline{w}}^{n+1}(\alpha)$.
\end{lemma}
\begin{proof}
Let $e_1'=g_1e_0'\in X_w(\alpha)$ have size $n$. Because $\tau$ normalizes $I_L$, left multiplication by $\tau$ is well defined bijection on $X_L'$. This action sends $e_1'$ to $g_1\tau e_0'$, an element of size $n+1$. The lemma follows since 
$ \inv(g_1\tau e_0',gg_1\tau e_0') = I_L\tau^{-1}g_1^{-1}gg_1\tau I_L = \tau I_L w I_L \tau^{-1}= I_L\tau^{-1} w \tau I_L = I\overline{w}I_L.$
\end{proof}

 The Iwahori-Hecke algebra $\mathcal{H}_L$ of $G_L$ is the convolution algebra of compactly supported $I_L$-bi-invariant functions from $G_L$ to $\mathbf{C}$, where by bi-invariant we mean $f(i_1gi_2)=f(g)$ for all $i_1$, $i_2\in I_L$. Of more importance to this paper is the center of this algebra with respect to convolution, denoted $\zhil$.
 
 We denote by $b:\zhie\rightarrow \zhi$ the \emph{Iwahoric base change homomorphism} \cite[p. 589]{haineslemma}. A description of the properties of $b$ relevant to our work will be postponed until section \ref{heckechapt}.



\section{Orbital Integral Preliminaries}\label{conjchapt}
\subsection{Conjugacy}
The following is a summary of results found in Chapter $4$ of \cite{langlands}. 
For every $\delta\in G_E$, $N(\delta)$ is $G_E$ conjugate to a $\gamma\in G$ unique up to $G$-conjugacy. We call $\gamma$ a \emph{norm} if it is conjugate to $N(\delta)$ for some $\delta\in G_E$. Conjugation, twisted conjugation, and the norm map are related by $g^{-1}N(\delta)g=N(g^{-1}\delta\sigma(g))$ and $\sigma(N(\delta))=\delta^{-1}N(\delta)\delta$.
\begin{lemma}\label{associated}
\begin{enumerate}
\item\label{assocone} Let $\delta\in G_E$ and let $\gamma\in G$ be conjugate to $N(\delta)$. There is a twisted conjugate $\delta'$ of $\delta$ such that $N(\delta')=\gamma$.
\item\label{assoctwo} Let $\gamma\in G$. Then $\gamma=N(\delta)$ for some $\delta\in G_E$ if and only if $\vd(\gamma)$ is divisible by $f$. 
\end{enumerate}
\end{lemma}
\begin{proof}
Part (\ref{assocone}) is \cite[pp. 32-33]{langlands}. For part (\ref{assoctwo}), corollary $4.7$ of \cite{langlands} tells us that $\gamma$ is a norm if and only if $\det(\gamma)\in N_{E/F}E^\times$. The lemma follows from basic local field theory. 
\end{proof}

\subsection{Elliptic Regular Semi-Simple Elements}\label{oldellipticsec}
It will turn out that we shall need to understand the action of semi-simple regular elliptic $\gamma$, and the $\delta$ which lie above them under the norm map, upon $\build(G_L)$ and $\build(H_L)$ (see lemma \ref{ellipticonly}).

For every semi-simple regular elliptic $\gamma$ in $G$ there is a basis of $F^2$ in which the matrix representation of $\gamma$ in $G$ is $$\gamma=\begin{pmatrix} x & yD\\ y & x\end{pmatrix}$$ for some $x,y, D\in F$, $y\neq 0$, $\sqrt{D}\notin F$, $\val(D)=1$ or $0$. The valuation of $D$ determines, and is determined by, the splitting field  $F'$ generated over $F$ by the eigenvalues $x\pm y\sqrt{D}$ of $\gamma$. This splitting field, over which $\mathbf{T}$ is a maximal split torus of $\mathbf{G}$, is isomorphic to $F[\sqrt{D}]$, and $F'$ is ramified over $F$ or not as $\val(D)$ is $1$ or $0$, respectively. We say a semi-simple regular elliptic element $g$ in such a form is in \emph{standard form}. If we denote by $\sigma_{F'}$ the unique non-trivial element of $\gal(F',F)$, then the eigenvalues of $\gamma$ can be rewritten $x+y\sigma_{F'}^i(\sqrt{D})$, $i=0,1$, and it is easy to see that both eigenvalues have the same valuation over $F'$.
When $\gamma$ is in standard form, the torus $T$ is $$T=\left\{\begin{pmatrix} \alpha & \beta D\\ \beta & \alpha\end{pmatrix}\in G\right\};$$ such an elliptic torus is also said to be in standard form. The maximal compact subgroup of such a $T$ is $T\cap G_\ints$.

\subsection{The Function $\Delta$}\label{ellipticsec}
Define the function $\Delta_L$ on the set $G_{L}^{\text{ss}}$ of semi-simple elements of $G_L$ as follows: let $g\in G_{L}^{\text{ss}}$ lie in a torus $\mathbf{T'}$ and have eigenvalues $\lambda_1$ and $\lambda_2$. Then
 $$\Delta(g)=\left|\frac{(\lambda_1-\lambda_2)^2}{\lambda_1\lambda_2}\right|_L^{1/2}=|\det(1-\adj(g));\mathfrak{g}/\mathfrak{t} |_L^{1/2},$$
 where $\mathfrak{g}=\lie(G_{\fbar})$ and $\mathfrak{t}=\lie(T'_{\fbar})$. The first form of the function, found in \cite[p. 48]{langlands}, will be of use to us computationally, while the second shows our function's relationship to functions of Harish-Chandra (cf. the function $D_{G(F)/M(F)(m)}$ in \cite[subsection 4.3]{haineslemma}). The eigenvalues of $g$ must lie in either $L$ or a quadratic extension $L'$. When they lie in $L$, it is obvious that the definition of $\Delta$ makes sense. If they lie in $L'$, they are $\sigma_{L'/L}$-conjugates and the argument of $||_L$ is $\sigma_{L'/L}$-invariant, so taking the $L$-norm also makes sense in that case.

\begin{lemma}\label{eigenvalues}
Suppose the eigenvalues $\lambda_1$ and $\lambda_2$ of $g$ lie in $L$, and that $\val_{L}(\lambda_1)=\val_{L}(\lambda_2)$. Then $\Delta_{L}(g) = q^{-a}$, $a\in \Z$ non-negative.
\end{lemma} 
\begin{proof}
By definition $\Delta_{L}(g) = q^{-\val_{L}((\lambda_1-\lambda_2)^2/\lambda_1\lambda_2)/2}$, so our goal is to show that $\val_{L(}(\lambda_1 - \lambda_2)^2/\lambda_1\lambda_2)/2$ is integral and non-negative. We have
\begin{align*}
\frac{\val_{L}\left(\frac{(\lambda_1-\lambda_2)^2}{\lambda_1\lambda_2}\right)}{2} 
&= \frac{2\val_{L}(\lambda_1-\lambda_2)-2\val_{L}(\lambda_1)}{2}\\
&= \val_{L}(\lambda_1-\lambda_2)-\val_{L}(\lambda_1),
\end{align*}
so the integrality is clear. Also from the hypothesis, $\val_{L}(\lambda_1-\lambda_2)\geq \val_{L}(\lambda_1)$, which implies $a\geq 0$. 
\end{proof}

\subsection{Orbital and Twisted Orbital Integrals}
 Let $f\in C_c(G)$, the set of compactly supported locally-constant $\C$-valued functions on $G$. The \emph{orbital integral} of $f$ with respect to $\gamma$ is  $\int_{T\backslash G} f(g^{-1}\gamma g)d\dot{g}$ and is denoted $\orb(f)$, where $d\dot{g}$ is the quotient measure of measures $dg$ and $dg_\gamma$, measures on $G$ and $G_\gamma$, respectively . Now let $f\in C_c(G_E)$. The \emph{twisted orbital integral} of $f$ with respect to $\delta\in G(E)$ is $\int_{G_{\delta}\backslash G(E)} f(g^{-1}\delta\sigma(g))d\dot{g}$ and is denoted $\twisted(f)$; here $d\dot{g}$ is the quotient measure of $dg$ and $dg_\delta$, a measure on $G_\delta$.




\section{The Buildings}\label{buildchapt}
Our calculations of orbital and twisted orbital integrals will depend primarily upon an understanding of, first, the buildings for $\mathbf{G}$ and $\mathbf{H}$ over $F$, $F'$, $E$, and the composite $EF'$ and, second, the action of semi-simple regular elliptic $\gamma$ and the $\delta$ which lie above them under $N$ upon those buildings.

\subsection{The Buildings for $SL_2$ and $GL_2$}\label{slbuildsec}
Let $L$ be a finite algebraic extension of $F$ with residual cardinality $q^n$. Recall the construction of the Bruhat-Tits building of ${H_{L}}$ \cite{brown,serre} and the extended Bruhat-Tits building of ${G_{L}}$ \cite{langlands}.
 The building $\build({H_{L}})$ is the tree with vertices the set of homothety classes of $\ints_L$-lattices, with two vertices $v_1$ and $v_2$ being adjacent (i.e., belonging to an edge) if there are lattices $\Lambda_i\in v_i$, $i=1,2$, such that $\Lambda_1\subsetneq \Lambda_2\subsetneq \pi_L^{-1}\Lambda_1$ and such that the quotients $\Lambda_2/\Lambda_1$ and $\pi_L^{-1}\Lambda_1/\Lambda_2$ are free $\ints_L/\pi_L\ints_L$-modules of rank $1$. The extended building $\build({G_{L}})$ is a tree with vertices $\ints_L$-lattices in $L^2$, and the same rule for determining edges as for $\build({H_{L}})$, with lattices instead of homothety classes of lattices. Passing from a lattice to its homothety class gives a simplicial map from $\build({G_{L}})$ to $\build({H_{L}})$. In either building each vertex is a facet of $q^n+1$ edges. 
 
  The \emph{type} of a vertex in $\build({H_{L}})$, an integer mod $2$, is determined as follows: let $v'=g\Lambda_0$ be a vertex in $\build({G_{L}})$. Then the type of $v'$ is $\val_L\det g \mod 2$. 
Type, for a lattice $\Lambda$, is constant on the whole homothety class of $\Lambda$, and so is also well defined for vertices in $\build({H_{L}})$. Adjacent vertices always have different types; when relevant in figures, we shall evince this by shading adjacent vertices differently.

 Let $L'$ be a quadratic extension of $L$, and let $\mathbf{T'}$ be a torus in $\mathbf{G}$ which splits, not over $L$, but over $L'$. The relationship between $\build({H_{L}})$ and $\build({H_{L'}})$ depends on whether or not $L'$ is ramified over $L$ \cite{kottwitz,langlands,rousseau}. In regards to buildings, when the word ``distance'' would be ambiguous, we will refer to distance in the building $\build(H_L)$ as $L$-distance and distance in $\build(H_{L'})$ as $L'$-distance.

\pagebreak[4]

 \subsubsection{$L'$ unramified over $L$}\label{lunram}
 \begin{center}
\begin{tikzpicture}
[axis/.style={<->,>=stealth'},every circle node/.style={draw,circle,
			inner sep=0mm},scale=.5,grow'=up,
	level 1/.style={level distance=2.0cm,sibling angle=120},
	level 2/.style={level distance=2.0cm,sibling angle=120,sibling distance=4cm},
	level 4/.style={level distance=2.0cm,sibling angle=120},
	level 5/.style={level distance=2.0cm,sibling angle=68},
	level 3/.style={level distance=2.0cm,sibling angle=57}]
  	\draw[white,myBG] (-8,0) -- (8,0);
  	\draw[black,] (-8,0) -- (8,0);
  	\draw (-8,0) node[left] {$\ldots$}; 
  	\draw (8,0) node[right] {$\ldots$};
  	\node at (11,0) {$\sapt_{T'}(L')$};  			
	  \foreach \pos in {0,2.5,5,7.5,-2.5,-5,-7.5}{
	    \node[fill=white] at (\pos,0) [circle, minimum size=1.8mm] {};
	    \node[fill=gray] at (\pos,0) [circle, minimum size=1.7mm] {};
	}
	    \node[fill=black] at (0,0) [circle, minimum size=1.7mm] {}; 
	    \node at (0,-.2) [below] {{\footnotesize$v_0$}};
	    \node[fill=black] at (0,0) [circle, minimum size=1.7mm] {} 
	    	child {node[fill=white] [circle,inner sep=0mm, minimum size=1.7mm] {}  
	    	    child {node[fill=black] [circle,inner sep=0mm, minimum size=1.7mm] {} edge from parent 
			child {node {$\vdots$} edge from parent 
			}
		     } 
		};
\end{tikzpicture}
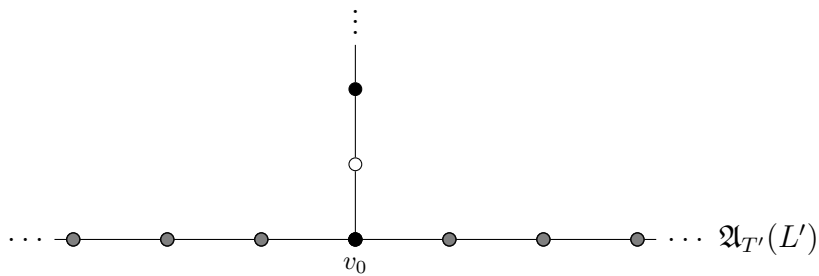
\captionme{The building $\build(H_L)$ inside $\build(H_{L'})$, $L'$ unramified over $L$}\label{unramfigone}
\end{center} 
 When $L'$ is unramified over $L$, $\build(H_L)$ lies in $\build(H_L')$ as the fixed point set of $\gal(L',L)$. There is a unique vertex shared by $\build(H_L)$ and $\sapt_{T'}(L')$ (the apartment associated to $T'_L$), and when $T'_L$ is in standard form that vertex is $v_0$, which  we defined to have type zero. 
Indeed, since $\sapt_{T'}(L')=g\sapt(L')$, where 
$$g=\begin{pmatrix}1 & -\frac{\sqrt{D}}{2}\\ \frac{1}{\sqrt{D}} & \frac{1}{2}\end{pmatrix},$$ an easy matrix computation shows that the vertex $v_0$ (which is of type zero) lies in $\sapt_{T'}(L')$ and is the only vertex in $\sapt_{T'}(L')$ fixed by $\gal(L',L)$.
  
In figure \ref{unramfigone} the gray vertices are facets strictly of $\build(H_{L'})$; we will follow this convention in all of our images. 

 \subsubsection{$L'$ ramified over $L$}\label{lram}
  When $L'$ is ramified over $L$, it is the first barycentric subdivision of $\build(H_L)$ that lies as a subtree of $\build(H_{L'})$, and it is no longer always true that $\build(H_L)$ is the entire Galois-fixed set, although $\build(H_L)=\build(H_{L'})^{\gal(L',L)}$ when $L'$ is tamely ramified over $L$ \cite[Prop. 5.1.1]{rousseau} and \cite{prasad}.
 The same analysis as in the unramified case shows that there is a unique Galois-fixed point, a vertex $v_{T'}$, in $\sapt_{T'}(L')$. Since $\build(H_L)$ is also Galois-fixed, one sees that $v_{T'}$ is also the unique point of $\sapt_{T'}(L')$ nearest $\build(H_L)$. Denote the point of $\build(H_L)$ nearest $\sapt_{T'}(L')$ (and so also nearest $v_{T'}$) by $p_{T'}$. In general the distance $d_{T'}$ between $p_{T'}$ and $v_{T'}$ can be greater than or equal to zero. 
We shall see in section \ref{ramsubsec} that $p_{T'}$ is the barycenter of an edge $e_T$ of $\build(H_L)$ and the picture for the action of $\gamma$ on $\build(H_L)$ looks like figure \ref{ramfigone}.
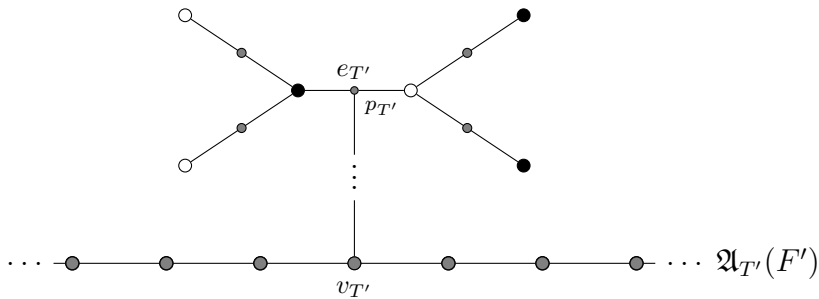
  \begin{center}
\begin{tikzpicture}
[axis/.style={<->,>=stealth'},every circle node/.style={draw,circle,
			inner sep=0mm},scale=.5,grow'=up,
	level 1/.style={level distance=2.5cm,sibling angle=120},
	level 2/.style={level distance=2.5cm,sibling angle=120,sibling distance=4cm},
	level 4/.style={level distance=2.5cm,sibling angle=120},
	level 5/.style={level distance=2.5cm,sibling angle=68},
	level 3/.style={level distance=2.5cm,sibling angle=57}]
  	\draw[white,myBG] (-8,0) -- (8,0);
  	\draw[black,] (-8,0) -- (8,0);
  	\draw (-8,0) node[left] {$\ldots$}; 
  	\draw (8,0) node[right] {$\ldots$};
  	\node at (11,0) {$\sapt_{T'}(F')$};  			
	  \foreach \pos in {2.5,5,7.5,-2.5,-5,-7.5}{
	    \node[fill=white] at (\pos,0) [circle, minimum size=1.8mm] {};
	    \node[fill=gray] at (\pos,0) [circle, minimum size=1.7mm] {};
	}
	    \node[fill=gray] at (0,0) [circle, minimum size=1.7mm] {}; 
	    \node[fill=gray] at (0,0) [circle, minimum size=1.7mm] {} 
	    	child {node {$\vdots$} edge from parent };
	    \node at (0,2.6) {}	
	       child[grow=up,level distance=2cm] {node [above] {{\footnotesize$e_{T'}$}} node [below right] {{\tiny$p_{T'}$}} node[fill=gray] [circle,inner sep=0mm, minimum size=1mm] {} 
	    	  child[grow=right,sibling distance=4cm,level distance=1.5cm] {node[fill=white] [circle,inner sep=0mm, minimum size=1.7mm] {} 
	    	  	child[level distance=3cm] {node[fill=black] [circle,inner sep=0mm, minimum size=1.7mm] {} 
	    	  	 edge from parent node [fill=gray] [circle,inner sep=0mm, minimum size=1.2mm] {} } 
	    	  	child[level distance=3cm] {node[fill=black] [circle,inner sep=0mm, minimum size=1.7mm] {} 
	    	  	 edge from parent node [fill=gray] [circle,inner sep=0mm, minimum size=1.2mm] {} }
	    	  } 
	    	  child[grow=left,sibling distance=4cm,level distance=1.5cm] {node[fill=black] [circle,inner sep=0mm, minimum size=1.7mm] {} 
	    	  	child[level distance=3cm] {node[fill=white] [circle,inner sep=0mm, minimum size=1.7mm] {} 
	    	  	 edge from parent node [fill=gray] [circle,inner sep=0mm, minimum size=1.2mm] {} } 
	    	  	child[level distance=3cm] {node[fill=white] [circle,inner sep=0mm, minimum size=1.7mm] {} 
	    	  	 edge from parent node [fill=gray] [circle,inner sep=0mm, minimum size=1.2mm] {} } 
	    	  }
		};
	    \node at (0,-.2) [below] {{\footnotesize$v_{T'}$}};
\end{tikzpicture}
\captionme{The building $\build(H_L)$ inside $\build(H_{L'})$, $L'$ ramified over $L$}\label{ramfigone}
\end{center}
\subsection{Computing Relative Distance In $\build(GL_2)$: Theory}\label{buildtheorysec}
The key to being able to emulate the work of Langlands in \cite{langlands} is the following theorem, which allows us to compute relative distance between an edge and its image under a simplicial map on $\build({G_{L}})$.

\begin{theorem}\label{distance}
Let $\alpha=(g,\sigma)\in G_L\times\gal(L,F)$; let $e'$ be an edge in $\build({G_{L}}$, and let $e$ be the image of $e'$ in $\build({H_{L}})$. Then $\inv(e',\alpha e')=(m,b,s)$, where $(m,b,s)$ is determined as follows.
\begin{enumerate}[\rm{(}a\rm{)}]
\item $s=\val_L\det(g)$.\label{distone}	
\item\label{disttwo} $b=1 \iff \ell(e,\alpha e)$ is odd.

\item\label{distfour} $\ell(e,\alpha e)= |2m+b|$
\item\label{distthree} If $e\neq \alpha e$, then $m< 0 \iff$ the vertex of $e$ farthest from $\alpha e$ is a vertex of type zero.
\end{enumerate}
\end{theorem}
Looking at the theorem statement, one can see that the computation breaks down into computing the $\Omega$ component and computing the $\wa$ component. In the course of the proof we shall indeed need to determine how to compute relative distance in $\build(H_L)$.
\begin{proof}
Let $e'=ge_0'$. Then $\inv(e',\alpha e')=Ig^{-1}\alpha gI=I_LwI_L=I_L(m,b,s)I_L$. The determinant of any element of $I_L$ has $L$-valuation 0, and $\val_L\det$ maps a product of matrices to the sum of the valuations of the determinants of the factors. Comparing the valuations of the determinants of factors of $Ig^{-1}\alpha gI$ to those of $I(m,b,s)I$ yields (\ref{distone}).

Part (\ref{distfour}) follows from lemma \ref{lengthlemma} and the definition of length on $\wae$. Part (\ref{disttwo}) is implied by (\ref{distfour}).

 All that remains is (\ref{distthree}).
Its proof requires a description of the relationship between the minimal gallery between two edges in $\build(H_L)$ and their relative distance and of the relationship between $\inv$ on $X_L'^2$ and $\inv$ on $X_L^2$ .
\begin{lemma}
Let $e_i\in X_L$, $i=1,2$. Then $w=\inv(e_1,e_2)$ may be characterized uniquely as the terminus of the type-preserving translation of the minimal gallery from $e_1$ to $e_2$ to a minimal gallery beginning at $e_0$ and contained entirely in $\sapt(L)$.
\end{lemma}
\begin{proof}
This result is \cite[p. 89, Exercise $3(b)$]{brown}.
\end{proof}

\begin{lemma}
Let $e_i'\in X'_L$, $i=1,2$, and let $e_i$ be the image of $e_i'$ under the map from $\build(G_L)$ to $\build(H_L)$. If $\inv(e_1',e_2')=(m,b,s)$ then $\inv(e_1,e_2)=(m,b)$.
\end{lemma}
\begin{proof}
Let $G_L$ act on the left of $X_L'\times X_L'$ by $g(e_1',e_2')=(ge_1',ge_2')$; $\inv$ is invariant under this action by definition. Since $\inv(e_1',e_2')=w$ is equivalent to the existence of some $g\in G_L$ such that $ge_1'=e_0'$ and $ge_2'=we_0'$, it will suffice to prove the lemma for $e_1'=e_0'$ and $e_2'=we_0'$. Then we need only show that the homothety class of $(m,b,s)e_0'$ is $(m,b)[e_0]$. But this is obvious since $(m,b,s)=(m,b,0)(0,0,s)$ and the homothety class of $(0,0,s)e_0'$ is clearly $[e_0']$. 
\end{proof}

Let $\Gamma_e$ be the minimal gallery from $e$ to $\alpha e$, let $\Gamma_w$ be the gallery from $e_0$ to $(m,b)e_0$, and let $g\in H_L$ be such that $g\Gamma_e=\Gamma_w$. 
Suppose that $e\neq \alpha e$, and let $v$ be the vertex of $e$ farthest from $\alpha e$. Then $gv$ is the farthest vertex of $e_0$ from $(m,b)e_0$. If $gv$ is type zero, then $(m,b)e_0$ must be to the right of $e_0$; if $gv$ is of type one, then $(m,b)e_0$ must be to the left of $e_0$. Since an edge of $\sapt$ is to the right of $e_0$ if and only if it is of the form $(m,b)e_0$ with $m<0$ (see, e.g., figure \ref{standardapt}), the theorem follows.
\end{proof}

\subsection{Computing Relative Distance In $\build(GL_2)$: Applications}\label{buildappsec}

Let $\gamma$ be semi-simple regular elliptic and let $\delta\in G_E$ be such that $N(\delta)=\gamma$. In this section we will investigate the actions of $\gamma$ and $\delta\sigma$ upon the buildings. Our goal is to use theorem \ref{distance} to determine those $w\in\wae$ such that $\xwgz$ and $\xwdz$ are non-empty, and to then determine the cardinality of those non-empty sets. 
These $\xwgz$ and $\xwdz$, along with the Bernstein coefficients to be discussed later, will completely determine the 
orbital and twisted orbital integrals on the center of the Iwahori-Hecke algebras. Since all the extended edges in $\xwgz$ and in $\xwdz$ have size zero, it suffices to count these edges by counting the the number of edges $e$ in the building for $H$ (resp. $H_E$) such that $\inv(e,\gamma e)$ (respectively $\inv(e, \delta\sigma(e))$) is the affine Weyl group part of $w$.

We will first study the $\xwgz$, and then the $\xwdz$. For each of $\gamma$ and $\delta\sigma$ we must break our analysis down into many sub-cases. Here is a diagram of these cases:
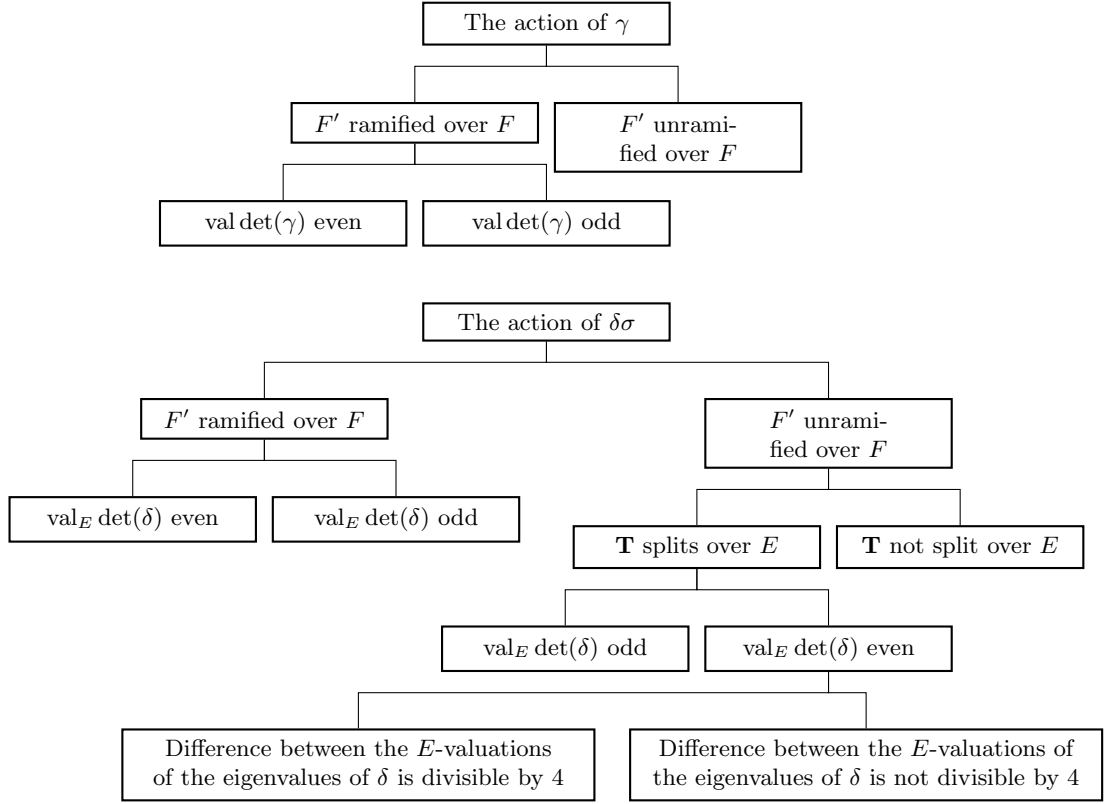
\begin{center}
  \begin{tikzpicture}[scale=.5,
    event/.style={rectangle,thick,draw,fill=white,text width=3cm,
    		text centered,anchor=north,font=\scriptsize},
    bigevent/.style={rectangle,thick,draw,fill=white,text width=6cm,
    		text centered,anchor=north,font=\scriptsize},
    edge from parent path={(\tikzparentnode.south) -- ++(0,-0.55cm)
    			-| (\tikzchildnode.north)},
	    level 1/.style={sibling distance=7cm,level distance=1.5cm,growth parent anchor=south,nodes=event},
	    level 2/.style={sibling distance=7cm,level distance=1.5cm}]
	\node at (0,0) [event] {The action of $\gamma$}
		child{node  {$F'$ ramified over $F$}
			child{node [event] {$\vd(\gamma)$ even}}
			child{node [event] {$\vd(\gamma)$ odd}}
		}
		child{node [event] {$F'$ unramified over $F$}
		};
	\node at (0,-8) [event] {The action of $\delta\sigma$} [sibling distance=12cm]
	  	child[sibling distance=15cm]{node [event] {$F'$ ramified over $F$}
			child{node [event] {$\val_E\det(\delta)$ even}}
			child{node [event] {$\val_E\det(\delta)$ odd}}
		}
		child[sibling distance=15cm]{node [event] {$F'$ unramified over $F$}
			child{node [event] {$\mathbf{T}$ splits over $E$}
				child{node [event] {$\val_E\det(\delta)$ odd}}
				child{node [event] {$\val_E\det(\delta)$ even}
					child[sibling distance=25cm]{node [bigevent] {Difference between the $E$-valuations of the eigenvalues of $\delta$ is divisible by $4$}}
					child[sibling distance=2cm]{node [bigevent] {Difference between the $E$-valuations of the eigenvalues of $\delta$ is not divisible by $4$}}
				}
			}
			child{node [event] {$\mathbf{T}$ not split over $E$}}
		};
  \end{tikzpicture}
\end{center}
Many of these cases will have similar analyses, and many will lead to similar or identical non-empty $\xwgz$ and $\xwdz$ and their cardinalities, but each case above will have to be addressed.

Our analysis here is based on that of \cite[Chapter 5]{langlands}.

\subsubsection{The action of $\gamma$}\label{gammaaction}
For the rest of this section, we denote $\val_F\det(\gamma)$ by $s$. Then \ref{distance} (\ref{distone}) says that $\inv(e',\gamma e')=(m,b,s)$ for any edge $e'\in\build(G_L)$ and for some $(m,b)\in \Z\times \Z/2\Z$. In order to compute $m$ and $b$, we study that action of $\gamma$ on $\build(H_L)$.

\begin{lemma}\label{langlandsfixed} Let $\mathbf{T}'$ be a torus in $\mathbf{G}$ split over $L$, let $g\in T'_L$ have distinct eigenvalues with identical $L$-valuations, and let $\Delta_L(g)=q^{-d}$. Then $g$, acting on $\build(H_L)$,  fixes precisely those points at a distance less than or equal to $d$ from $\sapt_{T'}(L)$\end{lemma}
Recall that $d$ is non-negative by lemma \ref{eigenvalues}.
\begin{proof}
This is \cite{langlands}, lemma 5.2.
\end{proof}

The eigenvalues of $\gamma$ are $\sigma_{F'/F}$-conjugates and so have the same valuation over $F'$. Thus $\gamma$ fixes precisely those edges at a distance less than or equal to $-\log_{q^n}(\Delta(\gamma))$ from $\apt(F')$, where $q^n$ is the cardinality of the residue field of $F'$ (here we are applying the lemma to the building $\build(H_{F'})$).

 \paragraph{\textbf{$F'$ ramified over $F$}}\label{ramsubsec}
We are now back to our earlier notation, where $\mathbf{T}$ is a torus defined over $F$ that splits over $F'$, and where $T=T(F)$. Recall from section \ref{lram} that $v_T$ is the unique Galois-fixed point of $\apt(F')$, $p_T$ is the unique point of $\build(H)$ nearest to $\apt(F')$ and to $v_T$, and $d_T$ is the $F'$-distance from $p_T$ to $v_T$. 
There exist $g\in T$ with $\val_F\det(g)$ odd. Such a $g$ does not act type-preservingly on the set of vertices in $\build(H)$ and so cannot fix any vertices of $\build(H)$. We have already seen, however, that $T$ fixes $p_T$; thus it is the barycenter of an edge $e_T$, which may also be characterized as the unique edge of $\build(H)$ closest to $\apt(F')$. Likewise, our chosen $\gamma$ will rotate $e_T$ about its barycenter if the $F$-valuation of its determinant is odd, and then the action of $\gamma$ will look like figure \ref{ramoddfig}.
\begin{center}
\begin{tikzpicture}
[axis/.style={<->,>=stealth'},every circle node/.style={draw,circle,
			inner sep=0mm},scale=.5,grow'=up,
	level 1/.style={level distance=2.5cm,sibling angle=120},
	level 2/.style={level distance=2.5cm,sibling angle=120,sibling distance=4cm},
	level 4/.style={level distance=2.5cm,sibling angle=120},
	level 5/.style={level distance=2.5cm,sibling angle=68},
	level 3/.style={level distance=2.5cm,sibling angle=57}]
  	\draw[white,myBG] (-5.5,0) -- (5.5,0);
  	\draw[black,] (-5.5,0) -- (5.5,0);
  	\draw (-5.5,0) node[left] {$\ldots$}; 
  	\draw (5.5,0) node[right] {$\ldots$};
  	\node at (8.5,0) {$\apt(F')$};  			
	  \foreach \pos in {2.5,5,-2.5,-5}{
	    \node[fill=white] at (\pos,0) [circle, minimum size=1.8mm] {};
	    \node[fill=gray] at (\pos,0) [circle, minimum size=1.7mm] {};
	}
	    \node[fill=gray] at (0,0) [circle, minimum size=1.7mm] {}; 
	    \node[fill=gray] at (0,0) [circle, minimum size=1.7mm] {} 
	    	child {node {$\vdots$} edge from parent };
	    \node at (0,2.6) {}	
	       child[grow=up,level distance=2cm] {node [above] {{\footnotesize$e_{T}=\gamma e_T$}} node [below right] {{\tiny$p_{T}$}} node[fill=gray] [circle,inner sep=0mm, minimum size=1mm] {} 
	    	  child[grow=right,sibling distance=4cm,level distance=1.75cm] {node[fill=white] [circle,inner sep=0mm, minimum size=1.7mm] {} 
	    	  	child[level distance=2cm] {node[fill=black] [circle,inner sep=0mm, minimum size=1.7mm] {} 
	    	  	 edge from parent node [fill=gray] [circle,inner sep=0mm, minimum size=1mm] {} 
			 node [above right] {$\gamma e$} }
	    	  	child[level distance=2cm] {node[fill=black] [circle,inner sep=0mm, minimum size=1.7mm] {} 
	    	  	 edge from parent node [fill=gray] [circle,inner sep=0mm, minimum size=1mm] {} }
	    	  } 
	    	  child[grow=left,sibling distance=4cm,level distance=1.75cm] {node[fill=black] [circle,inner sep=0mm, minimum size=1.7mm] {} 
	    	  	child[level distance=2cm] {node[fill=white] [circle,inner sep=0mm, minimum size=1.7mm] {} 
	    	  	 edge from parent node [fill=gray] [circle,inner sep=0mm, minimum size=1mm] {} 
			 node [above right] {$e$} } 
	    	  	child[level distance=2cm] {node[fill=white] [circle,inner sep=0mm, minimum size=1.7mm] {} 
	    	  	 edge from parent node [fill=gray] [circle,inner sep=0mm, minimum size=1mm] {} }
	    	  }
		};
\end{tikzpicture}\captionme{The action of $\gamma$ on $\build(H)$, $F'/F$ ramified, $\val_F\det(\gamma)$ odd}\label{ramoddfig}
\end{center}

When $\val_F\det(\gamma)$ is even, $\gamma$ acts type-preservingly; but $\gamma$ must send $e_T$ to itself,  and all type-preserving actions on $\build(H)$ sending $e_T$ to itself fix $e_T$. This and lemma \ref{langlandsfixed} imply that $\Delta_{F'}(\gamma)=q^{-b}$ for some $b\geq d_T+1$. Furthermore, $b-d_T-1$ must be even, otherwise $\gamma$ would fix exactly half of some edge in $\build(H)$, which is impossible. Write $b=2a+1+d_T$. Then $\Delta_F(\gamma)=q^{-a-\frac{d_T+1}{2}}$, and $\gamma$ fixes all those edges of $\build(H)$ an $F$-distance less than or equal to $a$ from $e_T$.
\begin{center}
\begin{tikzpicture}
[axis/.style={<->,>=stealth'},every circle node/.style={draw,circle,
			inner sep=0mm},scale=.5,grow'=up,
	level 1/.style={level distance=2.5cm,sibling angle=120},
	level 2/.style={level distance=2.5cm,sibling angle=120,sibling distance=4cm},
	level 4/.style={level distance=2.5cm,sibling angle=120},
	level 5/.style={level distance=2.5cm,sibling angle=68},
	level 3/.style={level distance=2.5cm,sibling angle=57}]
  	\draw[white,myBG] (-5.5,0) -- (5.5,0);
  	\draw[black,] (-5.5,0) -- (5.5,0);
  	\draw (-5.5,0) node[left] {$\ldots$}; 
  	\draw (5.5,0) node[right] {$\ldots$};
  	\node at (8.5,0) {$\apt(F')$};  			
	  \foreach \pos in {2.5,5,-2.5,-5}{
	    \node[fill=white] at (\pos,0) [circle, minimum size=1.8mm] {};
	    \node[fill=gray] at (\pos,0) [circle, minimum size=1.7mm] {};
	}
	    \node[fill=gray] at (0,0) [circle, minimum size=1.7mm] {}; 
	    \node[fill=gray] at (0,0) [circle, minimum size=1.7mm] {} 
	    	child {node {$\vdots$} edge from parent };
	    \node at (0,2.6) {}	
	       child[grow=up,level distance=2cm] {node [above] {{\footnotesize$e_{T}=\gamma e_T$}} node [below right] {{\tiny$p_{T}$}} node[fill=gray] [circle,inner sep=0mm, minimum size=1mm] {} 
	    	  child[grow=right,sibling distance=4cm,level distance=1.75cm] {node[fill=white] [circle,inner sep=0mm, minimum size=1.7mm] {} 
		    child[grow=south east,level distance=3cm] {node[fill=black] [circle,inner sep=0mm, minimum size=1.7mm] {} 
		      child[grow=east,level distance=1.5cm] {node {$\cdots$}
	    	  	child[grow=east,level distance=1.5cm] {node[fill=black] [circle,inner sep=0mm, minimum size=1.7mm] {} 
	    	  	child[grow=east,level distance=3cm] {node[fill=white] [circle,inner sep=0mm, minimum size=1.7mm] {} 
			  child[level distance=2cm,sibling distance=4cm] {node[fill=black] [circle,inner sep=0mm,minimum size=1.7mm] {}
	    	    		edge from parent node [fill=gray] [circle,inner sep=0mm, minimum size=1mm] {} 
				node [below left] {\footnotesize$\gamma \tilde{e}$} }
			  child[level distance=2cm] {node[fill=black] [circle,inner sep=0mm,minimum size=1.7mm] {}
	    	    		edge from parent node [fill=gray] [circle,inner sep=0mm, minimum size=1mm] {} 
				node [above left] {\footnotesize$\tilde{e}$} }
			edge from parent node [fill=gray] [circle,inner sep=0mm, minimum size=1mm] {} 
			node [above] {\footnotesize$\tilde{e}_a=\gamma \tilde{e}_a$}
			}
		        }
		      }
	    	    edge from parent node [fill=gray] [circle,inner sep=0mm, minimum size=1mm] {}
		    node [above right] {\footnotesize$\tilde{e}_1=\gamma \tilde{e}_1$}
		    }
	    	  } 
	    	  child[grow=left,sibling distance=4cm,level distance=1.75cm] {node[fill=black] [circle,inner sep=0mm, minimum size=1.7mm] {} 
		    child[grow=south west,level distance=3cm] {node[fill=white] [circle,inner sep=0mm, minimum size=1.7mm] {} 
		      child[grow=west,level distance=1.5cm] {node {$\cdots$}
	    	  	child[grow=west,level distance=1.5cm] {node[fill=white] [circle,inner sep=0mm, minimum size=1.7mm] {} 
	    	  	child[grow=west,level distance=3cm] {node[fill=black] [circle,inner sep=0mm, minimum size=1.7mm] {} 
			  child[level distance=2cm,sibling distance=4cm] {node[fill=white] [circle,inner sep=0mm,minimum size=1.7mm] {}
	    	    		edge from parent node [fill=gray] [circle,inner sep=0mm, minimum size=1mm] {} 
				node [above right] {\footnotesize$e$} }
			  child[level distance=2cm] {node[fill=white] [circle,inner sep=0mm,minimum size=1.7mm] {}
	    	    		edge from parent node [fill=gray] [circle,inner sep=0mm, minimum size=1mm] {} 
				node [below right] {\footnotesize$\gamma e$} }
			edge from parent node [fill=gray] [circle,inner sep=0mm, minimum size=1mm] {} 
			node [above] {\footnotesize$e_a=\gamma e_a$}
			}
		        }
		      }
	    	    edge from parent node [fill=gray] [circle,inner sep=0mm, minimum size=1mm] {}
		    node [above left] {\footnotesize$e_1=\gamma e_1$}
		    }
	    	  }
		};
\end{tikzpicture}\captionme{The action of $\gamma$ on $\build(H)$, $F'/F$ ramified, $\val_F\det(\gamma)$ even}
\end{center}
\begin{lemma}\label{ramcount} Let $F'$ be ramified over $F$, let $d_T$ be the distance from $p_T$ to $\apt(F')$, let $s=\val_F\det(\gamma)$, and let $\Delta(\gamma)=q^{-\left(a+\frac{d_T+1}{2}\right)}$.
\begin{enumerate}
\item\label{ramcountodd} When $s$ is odd, the non-empty $X_w^0(\gamma)$ are those with $w=(\pm r,0,s)$ for some $r\geq 0$, and the cardinality of $X_{(\pm r,0,s)}$ is $q^r$.
\item\label{ramcounteven} When $s$ is even, the non-empty $\xwgz$ are those with $w=(0,0,s)$, with cardinality $\dfrac{2q^{a+1}-q-1}{q-1}$, and $w=(-r,1,s)$ or $(r-1,1,s)$, $r\geq 1$, with cardinality $q^{a+r}$.
\end{enumerate}
\end{lemma}
\begin{proof}
First, let us suppose that $s$ is odd, and let $e$ be an edge in $\build(H)$ such that $\ell_F(e,e_T)$ is $r$, $r\geq 0$. Since $s$ is odd $\gamma$ fixes the barycenter $p_T$ of $e_T$ and no other point, permuting the vertices of $e_T$ that are facets over $F$. Thus $\ell_F(e,\gamma e)=2r$ is even, and so theorem \ref{distance} (\ref{disttwo}) implies that $b=0$ for all of these edges. When $r=0$ \ref{distance} (\ref{distfour}) implies $m=0$, so that $X_{(0,0,s)}(\gamma)$ is a singleton set, with $e_T$ being its only member. 

Now consider those edges with $r>0$. The action of $\gamma$ causes all such edges on one side of $e_T$ to move to edges on the other side, and vice versa, and so the farthest vertex of any such $e$ from $\gamma e$ must be the farthest vertex of $e$ from $e_T$.

Again, all those edges on one side of $e_T$ will have their farthest vertex from $e_T$ be of one type, while those vertices of the edges on the other side will have the other type. This and theorem \ref{distance} imply that those edges with $r>0$ are evenly distributed between the two sets $X_{(\pm r,0,s)}(\gamma)$, and that the cardinality of each set is the number of edges length $r$ away from $e_T$ and on one side of it. Let $v$ be a vertex of $e_T$. There are $q$ edges distance 1 from $e_T$ whose closest point of $e_T$ is $v$; there are $q^2$ edges at distance $2$ with the same property, and so on. In general, $\# X_{(\pm r, 0,s)}(\gamma)=q^r$.

Now suppose $s$ to be even. Then $\gamma$ fixes all edges of $\build(H)$ which lie at a length less than or equal to $a$ from $e_T$.

Theorem \ref{distance} implies that the fixed edges must lie in $X^0_{(0,0,s)}(\gamma)$; let us count them, by enumerating those at particular lengths from $e_T$. First is the edge $e_T$; then there are $2q$ distinct edges of $\build(H)$ attached to $e_T$, i.e., at length $1$ from $e_T$ ($2$ because we count edges radiating out of each $F$-vertex of $e_T$). At length $2$, there are $2q^2$, and, in general, at length $i\geq1$ there are $2q^i$ edges. Thus \begin{align*}
	\#X_{(0,0,s)}(\gamma) &= 1+2\sum_{i=1}^a q^i\\
	&= -1+2\sum_{i=0}^a q^i\\
	&= -1+2\frac{q^{a+1}-1}{q-1}\\
	&= \frac{2q^{a+1}-q-1}{q-1}.
\end{align*}

Consider those edges $e$ that are length $a+i$ away from $e_T$, $i\geq 1$.
 The length of the minimal gallery from $e$ to $\gamma e$ is $2i-1$, and therefore $b=1$. As in the case of odd $s$, all of the edges on one side of $e_T$ and length $a+i$ away from it have their farthest vertex from $\gamma e$ of one type, while those on the other side have farthest vertex of the other type, so that half of these edges have non-negative $m$ while half have negative $m$. Then theorem \ref{distance} tells us that the other non-empty $X^0_w(\gamma)$ are those with $w=(r-1,1,s)$ or $(-r,1,s)$, $r>0$, and the cardinality of $X_w^0(\gamma)$ for either $w$ is $q^{a+r}$.
\end{proof}

\paragraph{\textbf{$F'$ unramified over $F$}}
Suppose
 $F'/F$ is unramified and $\gamma$ (and therefore $T$) is in standard form.  Lemma
 \ref{langlandsfixed} and section \ref{slbuildsec} imply that if $\Delta(\gamma)=q^{-a}$, then $\gamma$ fixes precisely those points of $\build(H)$ a distance less than or equal to $a$ from $v_0$, the unique Galois-fixed point of $\sapt_{T}$.
\begin{center}
\begin{samepage}
\begin{tikzpicture}
[axis/.style={<->,>=stealth'},every circle node/.style={draw,circle,
			inner sep=0mm},scale=.5,grow'=up,
	level 1/.style={level distance=2.5cm,sibling angle=120},
	level 2/.style={level distance=2.5cm,sibling angle=120,sibling distance=4cm},
	level 4/.style={level distance=2.5cm,sibling angle=120},
	level 5/.style={level distance=2.5cm,sibling angle=68},
	level 3/.style={level distance=2.5cm,sibling angle=57}]
  	\draw[white,myBG] (-8,0) -- (8,0);
  	\draw[black,] (-8,0) -- (8,0);
  	\draw (-8,0) node[left] {$\ldots$}; 
  	\draw (8,0) node[right] {$\ldots$};
  	\node at (11,0) {$\apt(F')$};  			
	  \foreach \pos in {0,2.5,5,7.5,-2.5,-5,-7.5}{
	    \node[fill=white] at (\pos,0) [circle, minimum size=1.8mm] {};
	    \node[fill=gray] at (\pos,0) [circle, minimum size=1.7mm] {};
	}
	    \node[fill=black] at (0,0) [circle, minimum size=1.7mm] {}; 
	    \node at (0,-.2) [below] {{\footnotesize$v_0$}};
	    \node[fill=black] at (0,0) [circle, minimum size=1.7mm] {} 
	    	child {node[fill=white] [circle,inner sep=0mm, minimum size=1.7mm] {} 
			child {node {$\vdots$} edge from parent 
			} 
		edge from parent node [right] {{\footnotesize distance from $\apt(F')$ is $1$}}
		node [left] {{\footnotesize$e=\gamma e$}}
		};
	    \node at (0,5.1) {}
	    		child[level distance=2cm] {node[fill=white] [circle,inner sep=0mm, minimum size=1.7mm] {} 
	    		    child{node[fill=black] [circle,inner sep=0mm, minimum size=1.7mm] {} 
	    			child {node[fill=white] [circle,inner sep=0mm, minimum size=1.7mm] {}
					edge from parent node [left] {{\footnotesize$e$}}}
				child {node[fill=white] [circle,inner sep=0mm, minimum size=1.7mm] {}
					edge from parent node [right] {{\footnotesize$\gamma e$}}}
			    edge from parent node [right] {{\footnotesize distance from $\apt(F')$ is $a$}}
			    node [left] {{\footnotesize$e=\gamma e$}}
			    }
			};
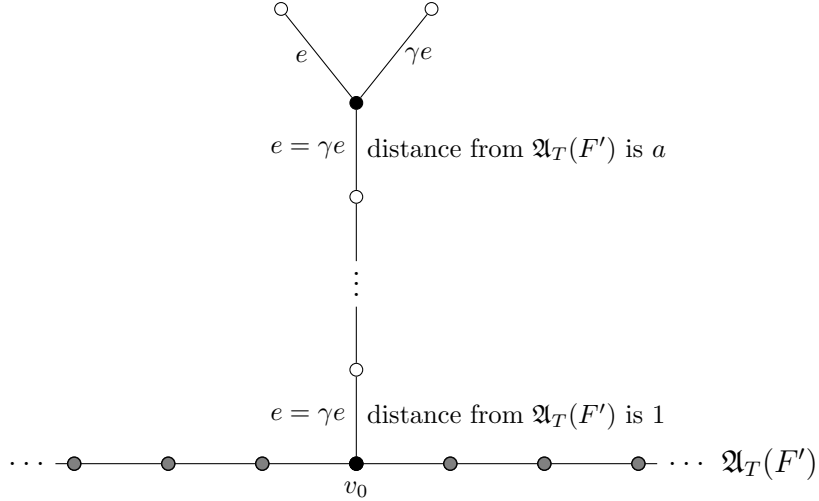
\end{tikzpicture}\captionme{The action of $\gamma$ on $\build(H)$, $F'/F$ unramified}
\end{samepage}
\end{center}

\begin{lemma}\label{unramcount}
Let $F'/F$ be unramified, let $s=\val_F\det(\gamma)$, and let $\Delta(\gamma)=q^{-a}$. The cardinality of $X^0_{(0,0,s)}(\gamma)$ is $\left(\frac{q+1}{q-1}\right)(q^a-1)$. When the type of $v_T$ is zero, 
\begin{align*}
 \#X_{(0,0,s)}^0(\gamma) &= \left(\frac{q+1}{q-1}\right)(q^a-1)\\
 \#X_{(-r,1,s)}^0(\gamma) &= \begin{cases} (q+1)q^{a+r-1}, & a+r\equiv0\mod 2\\
 				 0, & a+r\equiv 1\mod 2\end{cases}\\
 \#X_{(r-1,1,s)}^0(\gamma) &= \begin{cases} (q+1)q^{a+r-1}, & a+r\equiv 1\mod 2\\
 				 0, & a+r\equiv 0\mod 2\end{cases}
\end{align*}
where $r\geq 1$ is an integer.
\end{lemma}

\begin{proof}
 Every $\gamma$-fixed edge $e$, $\inv(e,\gamma e)=(m,b,s)$, has $m=b=0$. If $a=0$ then the only $\gamma$-fixed point of $\build(H)$ is $v_0$ and $X^0_{(0,0,s)}(\gamma)$ is empty, which agrees with our formula above. When $a>0$ there are $q+1$ edges in $\build(H)$ attached to $v_0$ (length $1$ from $\apt(F')$), and, branching outwards, $q$ edges in $\build(H)$ attached to each of those, and so on; in total there are  
$$\sum_{i=1}^a (q+1)q^{i-1} = \frac{q+1}{q-1}(q^a - 1)$$ fixed
edges in $\build(H)$, all of which lie in $X^0_{(0,0,s)}(\gamma)$.

For those edges lying at distance $k=a+r$, $r\geq 1$, from $\apt(F')$, $\ell(e, \gamma e)$ must be odd; for $e$ is length $r$ away from the nearest $\gamma$-fixed edge, as is $\gamma e$, so that $\ell(e, \gamma e)=2r-1$. Theorem \ref{distance} (\ref{disttwo}) implies that these edges, we have $b=1$. 

The vertex $v_0$ has type zero, so the type of the vertex of $e$ farthest from $v_0$ is $0$ for those with even $k$ and $1$ for those and odd $k$. The farthest vertex of $e$ from $v_0$ is also the first vertex in the minimal length gallery from $e$ to $\gamma e$, and 
since $k$ is even or odd as $a$ and $r$ share or do not share parity, theorem \ref{distance} (\ref{distthree}) implies that $m<0$ when $a$ and $r$ share parity, while $m\geq0$ when $a$ and $r$ have different parities. This and theorem \ref{distance} (\ref{distfour}), which gives $2r-1=|2m+1|$, imply that $m=-r$ when $a$ and $r$ share parity and $m=r-1$ when they do not. 

Since there are $(q+1)q^{a+r-1}$ edges in $\build(G)$ at length $k=a+r$ from $\apt(F')$, the lemma follows.
\end{proof}

\subsubsection{The action of $\delta\sigma$}\label{deltaaction}
Let $\delta\in G_E$ and let $e'=ge_0'$ be an extended edge in $\build(G_E)$. Then  we may write $\inv(e',\delta\sigma(e')) = (m,b,s)$, where $(m,b,s)$ is defined by $I_Eg^{-1}\delta\sigma(g)I_E=I_E(m,b,s)I_E$. Comparing the valuations of the determinants of representatives on either side of the equation shows us that $s=\val_E\det\delta$. The majority of the rest of this section is devoted to determining those $(m,b)$ such that $X_{(m,b,s)}^0(\delta\sigma)$ is non-empty, and the cardinality of the non-empty \xwdz. As in our analysis of the action of $\gamma$, the rest of our analysis may be performed on $\build(H_E)$.

First, however, there is a more illuminating description of the twisted torus $G_\delta$ in the scenario of greatest interest to us, as well as an important lemma about the action of $\delta$ on $\build(H_E)$.
\begin{lemma}\label{toruslemma}
Let $\delta\in G_E$ be such that $\gamma=N(\delta)$ is semi-simple regular. Then 
\begin{align*}
G_\delta(F) &= T.
\end{align*}
\end{lemma}
 \begin{proof}
 We show this result by showing that both $T$ and $G_\delta(F)$ are equal to $T_\delta(F)$, where $T_\delta(F)=T(E)\cap G_\delta(F)$. We first show that $G_\delta(F)=T_\delta(F)$. The definition of $T_\delta(F)$ says that $T_\delta(F)\subseteq G_\delta(F)$. Let $g\in G_\delta(F)$. By definition $g^{-1}\delta=\delta\sigma(g^{-1})$, which implies that for any $i\in \Z$, $\sigma^i(g^{-1}\delta)=\sigma^i(\delta)\sigma^{i+1}(g^{-1})$. Then
\begin{align*}
g^{-1}\gamma g &= g^{-1}N(\delta)g\\
 &= N(g^{-1}\delta\sigma(g))\\
 &= N(\delta)\\
 &= \gamma.
\end{align*}
This shows that $G_\delta(F)\subseteq T$, hence $G_\delta(F)\subseteq T_\delta(F)$.

 Now we must show that  $T\sigd(F)=T(F)$. We know that $T_\delta(F)\subseteq T(E)$ by construction.
 Let $t\in T_\delta(F)$, so $\delta=t^{-1}\delta\sigma(t)$. Then $\delta\in T(E)$, since $\gamma=\sigma(\gamma)=\sigma(N(\delta)=\delta^{-1}N(\delta)\delta = \delta^{-1}\gamma\delta$. Because $\mathbf{T}$ is abelian we must have $1=t^{-1}\sigma(t)$, that is, $t\in T(F)$. 

Finally, let $t\in T(F)$. Then 
 \begin{align*}
 t^{-1}\delta\sigma(t)&= t^{-1}\sigma(t)\delta\\
 &= \delta.
 \end{align*}
 \end{proof}

\begin{lemma}\label{langlemmatwo}
Suppose that the eigenvalues of $\gamma$ have identical valuations and let $\Delta(\gamma)=q^{-d}$, so that $\gamma$ fixes precisely those points of $\build(H_{EF'})$ an $EF'$-distance less than or equal to $d$ from $\apt(EF')$. Let $\delta$ be chosen such that $N\delta=\gamma$. Then 
\begin{enumerate}
\item\label{deltafixedone} $\delsig$ fixes a vertex $v$ in $\build(H_E)$ only if $\gamma$ does.
\item\label{deltafixedtwo} If $\delsig$ fixes a vertex $v$ of $\build(H_E)$ an $EF'$-distance less than $d$ from $\apt(EF')$ in $\build(H_{EF'})$, then the number of edges of $\build(H_{E})$ of which $v$ is a facet and which are also fixed by $\delsig$ is $q+1$ (the same as the number of edges which can be joined to any vertex in $\build(G)$) 
\end{enumerate}
\end{lemma}
\begin{center}
\begin{tikzpicture}
[axis/.style={<->,>=stealth'},every circle node/.style={draw,circle,
			inner sep=0mm},scale=.5,grow cyclic,
	level 1/.style={level distance=6cm,sibling angle=72},
	level 2/.style={level distance=4cm,sibling angle=72},
	level 3/.style={level distance=2cm,sibling angle=72}]
	\node at (0,0) [above right] {\scriptsize$v=\delta\sigma v$};    
	\node[fill=black] at (0,0) [circle, minimum size=1.7mm] {}
    	  child {node[fill=white] [circle,inner sep=0mm, minimum size=1.7mm] {} 
	  edge from parent node [above left] {\footnotesize$e_4=\delta\sigma e_4$}
	  }
    	  child {node[fill=white] [circle,inner sep=0mm, minimum size=1.7mm] {} 
	  edge from parent node [right] {\footnotesize$\delta\sigma e_3$}
	  }
    	  child {node[fill=white] [circle,inner sep=0mm, minimum size=1.7mm] {} 
    	    child {node[fill=black] [circle,inner sep=0mm, minimum size=1.7mm] {} 
	    edge from parent node [below right] {\scriptsize$e_8=\delta\sigma e_8$}
	    }
    	    child {node[fill=black] [circle,inner sep=1mm, minimum size=1.7mm] {} 
	    edge from parent node [right] {\scriptsize$\delta\sigma e_7$}
	    }
    	    child[level distance=6cm] {node[fill=black] [circle,inner sep=0mm, minimum size=1.7mm] {} 
    	  	child {node[fill=white] [circle,inner sep=0mm, minimum size=1.7mm] {} 
	    	edge from parent node [below left] {\tiny$\delta\sigma e_9$}
	  	}
    	  	child {node[fill=white] [circle,inner sep=0mm, minimum size=1.7mm] {} 
	    	edge from parent node [below] {\tiny$\delta\sigma e_{10}$}
	  	}
    	  	child {node[fill=white] [circle,inner sep=0mm, minimum size=1.7mm] {} 
	    	edge from parent node [right] {\tiny$e_{10}$}
	  	}
    	  	child {node[fill=white] [circle,inner sep=0mm, minimum size=1.7mm] {} 
	    	edge from parent node [below left] {\tiny$e_9$}
		}
	    edge from parent node [below right] {\scriptsize$e_6=\delta\sigma e_6$}
	    }
    	    child {node[fill=black] [circle,inner sep=0mm, minimum size=1.7mm] {} 
	    edge from parent node [above right] {\scriptsize$e_7$}
	    }
	  edge from parent node [below] {\footnotesize$e_1=\delta\sigma e_1$}
	  }
    	  child {node[fill=white] [circle,inner sep=0mm, minimum size=1.7mm] {} 
	  edge from parent node [right] {\footnotesize$e_2=\delta\sigma e_2$}
	  }
    	  child {node[fill=white] [circle,inner sep=0mm, minimum size=1.7mm] {} 
	  edge from parent node [above right] {\footnotesize$e_3$}
	  };
\end{tikzpicture}
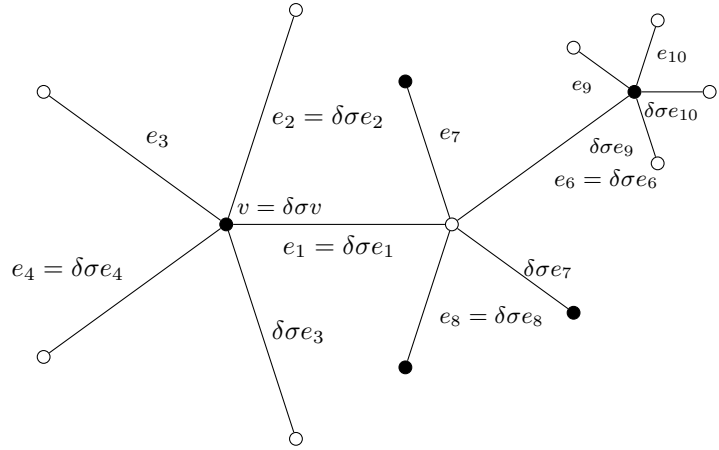
\captionme{The action of $\delsig$ on edges attached to a $\delsig$-fixed vertex $v\in\apt(EF')$, $q=f=a=2$}
\end{center}
\begin{proof}
Let $\alpha$ act on $\build(H_L)$ by $\alpha e = \delta\sigma(e)$. Since $\sigma^f$ acts as the identity on $\build(H_E)$, the action of $\alpha^f$ on $\build(H_L)$ is clearly the same as the action of $N(\delta)=\gamma$, and part (\ref{deltafixedone}) follows. Part (\ref{deltafixedtwo}) is proven in \cite[pp. $66$--$67$]{langlands}.
\end{proof}

 Our next task is to find the non-empty $\xwdz$ and compute their cardinalities.  Recall that the residue field of $E$ has cardinality $q^f$.
 
\paragraph{\textbf{$F'$ ramified over $F$}}
When $F'$ is ramified over $F$, $T$ never splits in $E$, but does in $EF'$. The relationship between $\build(H_E)$ and $\build(H_{EF'})$ is exactly analogous to that between $\build(H)$ and $\build(H_{F'})$. There is a unique point $p_T$ in $\build(H_E)$, the barycenter of an edge $e_T$ in $\build(H_E)$. This $p_T$ is a vertex of $\apt(EF')$.  Compare the following to lemma \ref{ramcount}.
 
\begin{lemma}\label{twramcount}Let $F'/F$ be ramified, let $s=\val_E(\det(\delta))$, let $\delta\in G_E$ be such that $N(\delta)=\gamma$ is semi-simple regular elliptic, and let $\Delta(\gamma)=q^{-a-\frac{d_T+1}{2}}$.
\begin{enumerate}
\item\label{twramcountone} If $s$ is odd, then the non-empty $\xwdz$ are those with $w=(\pm r,0,s)$, $r\geq 0$, with cardinality $q^{fr}$.
\item\label{twramcounttwo} If $s$ is even, then the non-empty $\xwdz$ are those with $w=(0,0,s)$, with cardinality $\frac{2q^{a+1}-q-1}{q-1}$, and $w=(r-1,1,s)$, or $(-r,1,s)$, $r\geq 1$,with cardinality $q^{fr}\left[q^a+\frac{q^{a}-1}{q-1}(1-q^{1-f})\right]$.
\end{enumerate}
\end{lemma}
\begin{proof}
We have already shown above that the size of any $w$ such that $\xwdz$ is non-empty must be $s$. Suppose first that $s$ is odd. The apartment $\apt(F')$ is the same as the apartment $\apt(EF')$ because $EF'$ is unramified over $F'$, and the building $\build(H_{F'})$ lies within $\build(H_{EF'})$ as the fixed point set of $\gal(EF',F')$. 
In addition to $e_T$ there is its $\build(H)$ analogue, the edge $\tilde{e}_T$ in $\build(H)$ closest to $\apt(F')$ within $\build(H_{F'})$. We must have $e_T=\tilde{e}_T$, as $\tilde{e}_T$ is also in $\build(H_E)$ and they share their barycenter.

Since $e_T$ lies in $\build(H)$, $\sigma$ fixes $e_T$ point-wise. We know that $\delta$ stabilizes $e_T$ for the same reasons that $\gamma$ stabilized the edge $e_T$ in section \ref{ramsubsec} . When $s$ is odd, $\delta$ has no fixed point in $\build(H_{E})$, so it must flip $e_T$ end to end, and we find ourselves yet again in the same situation as in the proofs of lemmas \ref{ramcount} (\ref{ramcountodd}),
 and (\ref{twramcountone}) is proven.

When $s$ is even, $\delta$ moves vertices an even distance, and so it and $\delsig$ must fix $e_T$. Lemma \ref{langlemmatwo} (\ref{deltafixedtwo}) and our analysis from section \ref{ramsubsec} tell us that attached to each vertex of $e_T$ are $q$ $\delsig$-fixed edges at distance $1$ from $e_T$, $q^{2}$ such edges at distance $2$, and, in general, $q^{i}$ edges at distance $i$, for all $1\leq i\leq a$. Added all together, this yields a total of 
\begin{align*}
1+2\sum_{i=1}^{a}q^{i} &= 2\sum_{i=0}^a q^{i} -1\\
&= \frac{2(q^{a+1}-1)}{q-1}-1\\
&= \frac{2q^{a+1}-q-1}{q-1}
\end{align*}
fixed edges. We know that all fixed edges have relative distance $(0,0,s)$, yielding the first half of (\ref{twramcounttwo}).

Now fix $r\geq 1$; we wish to find all possible relative distances $\inv(e,\delsig(e))$  of $e$ distance $r$ away from $\build(H_{E})^{\delsig}$. This situation is similar to that in the proof of lemma \ref{ramcount} (\ref{ramcounteven}). We see immediately that: $\ell(e,\delsig e)=2r-1$, so that $b=1$; the minimal gallery from $e$ to $\delsig e$ begins with the farthest vertex of $e$ from $\build(H_{E})^{\delsig}$; exactly half such edges have their farthest vertex from $\build(H_{E})^{\delsig}$ of type one, while the other half has type zero; and therefore $w=\inv(e,\delsig(e))=(-r,1,s)$ or $(r-1,1,s)$, with half of the edges at distance $r$ from $\build(H_{E})^{\delsig}$ of one relative distance and half of the other. Thus, to determine the cardinality of $X^0_w(\delsig)$ for either $w$, we need only count all the edges on one side of $e_T$ that are $r$ away from $\build(H_{E})^{\delsig}$.

Our main tool is lemma \ref{langlemmatwo} (\ref{deltafixedtwo}). On one side of $e_T$, there are $(q^f-q)$ edges attached to it that are not $\delsig$-fixed, and attached to each of those there are $q^{f(r-1)}$ edges distance $r$ from $e_T$, for a total of $q^{fr}(1-q^{1-f})$ edges whose nearest $\delsig$-fixed point is in $e_T$. In the same manner there are $q^i$ $\delsig$-fixed edges distance $0\leq i<a$ from $e_T$, and for each such edge $e'$ there are $q^{fr}(1-q^{1-f})$ edges whose nearest $\delsig$ fixed point is in $e'$ and which are distance $r$ from $e'$, for a total of $q^iq^{fr}(1-q^{1-f})$ edges a distance $r$ from $\build(H_{E})^{\delsig}$ and whose nearest point in $\build(H_{E})^{\delsig}$ is distance $i$ from $e_T$. Since there are no $\delsig$-fixed edges at a distance greater than $a$ from $e_T$, there are $q^aq^{fr}$ edges distance $r$ from $\build(H_{E})^{\delsig}$ whose nearest $\delsig$-fixed edge is distance $a$ from $e_T$. Added all together, this gives us, for $w=(-r,1,s)$ or $(r-1,1,\vd(\delsig))$,
\begin{align*}
\#X^0_w(\delsig)&= q^aq^{fr}+\sum_{i=0}^{a-1}q^i(1-q^{1-f})q^{fr}\\
&= q^{fr}\left[q^a+(1-q^{1-f})\frac{q^a-1}{q-1}\right].
\end{align*}
\end{proof} 
 
\paragraph{\textbf{$F'$ unramified over $F$}}
Compare the following to lemma \ref{unramcount}.
\begin{lemma}\label{twunramcount}Let $F'/F$ be unramified, let $\gamma$ be in standard form, let $s=\val_F\det(\gamma)$, and let $\Delta_F(\gamma)=q^{-a}$.
\begin{enumerate}
\item\label{twunramcountone} Suppose that $\gamma$ splits in $E$ and has $s$ odd. Then the non-empty $\xwdz$ are those with $w=(\pm r,0,s)$, $r\geq 0$, with cardinality $q^{fr}$.
\item\label{twunramcounttwo} Suppose that $\gamma$ either splits in $E$ with $s$ even or $\gamma$ does not split in $E$. Then the non-empty $\xwdz$ are those with $w=(0,0,s)$, $(r-1,1,s)$, or $(-r,1,s)$, $r\geq 1$. In either case the cardinality of $X_{(0,0,s)}^0(\delsig)$ is $\left(\frac{q+1}{q-1}\right)(q^a-1)$. When $\gamma$ does not split in $E$, we have
\begin{align*}
\#X_{(-r,1,s)}^0(\delta\sigma) &= \begin{cases} q^{fr}\left[(1-q^{1-f})\frac{q^{a-1}-1}{q-1}+(q+1)q^{a-1}\right], & a+r\equiv0\mod 2\\
 				 q^{fr}(1-q^{1-f})\frac{q^a-1}{q-1}, & a+r\equiv 1\mod 2\end{cases}\\
 \#X_{(r-1,1,s)}^0(\delta\sigma) &= \begin{cases}
  q^{fr}\left[(1-q^{1-f})\frac{q^{a-1}-1}{q-1}+(q+1)q^{a-1}\right], & a+r\equiv 1\mod 2\\
 				 q^{fr}(1-q^{1-f})\frac{q^a-1}{q-1}, & a+r\equiv 0\mod 2\end{cases}
\end{align*}
When $\gamma$ splits in $E$ and $s$ is even, then the cardinalities of $X_{(r-1,1,s)}^0(\delsig)$ and $X_{(-r,1,s)}^0(\delsig)$ are the same as the above if the difference of the valuations of the eigenvalues of $\delta$ is divisible by $4$; otherwise we have
\begin{align*}
\#X_{(r-1,1,s)}^0(\delta\sigma) &= \begin{cases} q^{fr}\left[(1-q^{1-f})\frac{q^{a-1}-1}{q-1}+(q+1)q^{a-1}\right], & a+r\equiv0\mod 2\\
 				 q^{fr}(1-q^{1-f})\frac{q^a-1}{q-1}, & a+r\equiv 1\mod 2\end{cases}\\
 \#X_{(-r,1,s)}^0(\delta\sigma) &= \begin{cases}
  q^{fr}\left[(1-q^{1-f})\frac{q^{a-1}-1}{q-1}+(q+1)q^{a-1}\right], & a+r\equiv 1\mod 2\\
 				 q^{fr}(1-q^{1-f})\frac{q^a-1}{q-1}, & a+r\equiv 0\mod 2\end{cases}
\end{align*}
\end{enumerate}
\end{lemma}
\begin{proof}
 Let us handle (\ref{twunramcountone}) first. 
The apartments $\apt(F')$ and $\apt(E)$ are the same (since $E$ is unramified over $F'$). Then $\sigma$ acts non-trivially on $\apt(F')=\apt(E)$ since $\sigma^{f/2}$ generates $\gal(E,F)$. The only non-trivial Galois action on $\apt(F')$ is flipping it end to end, and so $\sigma$ does the same to $\apt(E)$. In the case under consideration $s$ is odd, so that the valuation over $E$ of the eigenvalues of $\delta$ cannot be equal. We know from above that $\delta\in T(E)$, and so $\delta$ acts on $\apt(E)$ by shifting it an odd distance left or right. Putting these actions together, it is easy to see that $\delsig$ fixes a single point $p$ of $\apt(E)$, and that $p$ is not a vertex in the building $\build(H_{E})$. Let $e_T$ be the edge to which $p$ belongs. Then $e_T$ is flipped end to end, and our result follows from the proof of lemma \ref{ramcount} (\ref{ramcountodd}).

The formulae in (\ref{twunramcounttwo}) are more complicated, owing entirely to the structures of the buildings involved and lemma \ref{langlemmatwo}.
We have two cases to consider; the hard work in both cases is identical, and stems from showing that the action of $\delsig$ fixes precisely one point in $\build(H_{E})$ that also lies in $\apt$ (over an appropriate field), and that this point is a vertex. Let us postpone the proof that both cases possess this property, and simply assume that we know that this is the case. Let us call this sole fixed point $v$. Then lemma \ref{langlemmatwo} tells us that there are exactly $q+1$ edges in $\build(H_{E})$ with $v$ as a vertex that are also fixed by $\delsig$. For the same reason each of those edges is adjacent to $q$ different edges (at a distance $2$ from $\apt$) which are $\delsig$-fixed; and so on. Comparison to our earlier work in the proof of \ref{unramcount} shows us that all of these fixed edges lie in $X_{(0,0,s)}^0(\delsig)$, and that there are exactly $\left(\frac{q+1}{q-1}\right)(q^a-1)$ of them. 

Now we must count the non-fixed edges. Let $e$ be such an edge, lying at distance $r\geq 1$ from the nearest $\delsig$-fixed edge. Our earlier work in the proof of lemma \ref{unramcount} on the possible $w=\inv(e,\gamma e)$ when $e$ is distance $r$ from the nearest $\gamma$-fixed edge applies identically in this case, showing us that $w=(r-1,1,s)$ or $(-r,1,s)$, depending on whether the farthest vertex of $e$ from $\build(H_{E})^{\delsig}$ is type one or type zero, respectively. Call this vertex $v_e$, and denote by $k_e$ the distance from $v_e$ to $v$.

Let $w_1=(r-1,1,s)$ and $w_2=(-r,1,s)$. If $v$ has type zero, then when $k_e$ is odd, $v_e$ has type one, and so $e$ lies in $X^0_{w_1}(\delsig)$; when $e_k$ is even, $v_e$ has type zero and $e$ lies in $X^0_{w_2}(\delsig)$. When $v$ has type $1$, then the opposite occurs; that is, if $\# X^0_{w_1}(\delsig)=x$ and $\# X^0_{w_2}(\delsig)=y$ when $v$ has type zero, then $\# X^0_{w_1}(\delsig)=y$ and $\# X^0_{w_2}=x$ when $v$ has type one. Thus proving the lemma for $v$ type $0$ and showing that the type of $v$ is one if and only if $\gamma$ splits in $E$ and $4$ does not divide the difference of the valuations of the eigenvalues of $\delta$ will suffice to finish the lemma.

Consider first those edges who nearest $\delsig$-fixed point is $v$.
  There are $q^f+1$ edges in $\build(H_{E})$ with $v$ as a facet; lemma \ref{langlemmatwo} tells us that precisely $q+1$ of them are fixed, leaving exactly $q^f-q$ not-fixed edges at distance one from $v$. Being distance one from $v$, which has type zero, all of these edges must lie in $X^0_{w_1}(\delsig)$. 

Attached to the $q^f-q$ not-$\delsig$-fixed edges at distance one from $v$ are a total of  $(q^f-q)q^f$ edges, all of which must lie in $X^0_{w_2}(\delsig)$. Continuing with this line of reasoning, we see that, altogether, there are $q^{fr}(1-q^{1-f})$ edges in the tree at a distance $r$ from $v$ and whose nearest $\delsig$-fixed point is $v$, and these edges lie $X^0_{w_1}(\delsig)$ if $r$ odd and $X^0_{w_2}(\delsig)$ if $r$ is even.
	
\begin{center}
\begin{tikzpicture}
[	sibling angle=36,
	axis/.style={<->,>=stealth'},every circle node/.style={draw,circle,
			inner sep=0mm},scale=.5,grow'=up,
	level 1/.style={level distance=2.5cm,sibling distance=2cm},
	level 2/.style={level distance=2.5cm,grow cyclic},
	level 4/.style={level distance=2.5cm},
	level 5/.style={level distance=2.5cm},
	level 3/.style={level distance=2.5cm}]
  	\draw[white,myBG] (-8,0) -- (8,0);
  	\draw[black,] (-8,0) -- (8,0);
  	\draw (-8,0) node[left] {$\ldots$}; 
  	\draw (8,0) node[right] {$\ldots$};
  	\node at (11,0) {$\apt$};  			
	  \foreach \pos in {0,2.5,5,7.5,-2.5,-5,-7.5}{
	    \node[fill=white] at (\pos,0) [circle, minimum size=1.8mm] {};
	    \node[fill=gray] at (\pos,0) [circle, minimum size=1.7mm] {}; 
	}
	    \node[fill=black] at (0,0) [circle, minimum size=1.7mm] {}; 
	    \node at (0,-.2) [below] {{\footnotesize$v$}};
	    \node[fill=black] at (0,0) [circle, minimum size=1.7mm] {} 
		child {node [fill=white] [circle,inner sep=0mm, minimum size=1.7mm] {} 
		   child[rotate=180]{node [thin,fill=black] [circle,inner sep=0mm, minimum size=1.7mm] {}
		   edge from parent[thin] node [above] {{\footnotesize$\delsig e_1$}}
		   }
		   child[rotate=180]{node [thin,fill=black] [circle,inner sep=0mm, minimum size=1.7mm] {}
		   edge from parent[thin] node [below] {{\footnotesize$e_1$}}
		   }
		edge from parent[line width=1mm]
		}
		child {node [thin,fill=white] [circle,inner sep=0mm, minimum size=1.7mm] {}
		   child[rotate=120]{node [thin,fill=black] [circle,inner sep=0mm, minimum size=1.7mm] {} 
			child {node [thin,fill=white] [circle,inner sep=0mm, minimum size=1.7mm] {} 
		   	edge from parent[thin] node [right] {{\footnotesize$\delsig e_3$}}
			}
			child {node [thin,fill=white] [circle,inner sep=0mm, minimum size=1.7mm] {} 
		   	edge from parent[thin] node [left] {{\footnotesize$e_3$}}
			}
		   }
		   child[rotate=120]{node [thin,fill=black] [circle,inner sep=0mm, minimum size=1.7mm] {}
			child {node [thin,fill=white] [circle,inner sep=0mm, minimum size=1.7mm] {} 
		   	edge from parent[thin] node [right] {{\footnotesize$\delsig e_2$}}
			}
			child {node [thin,fill=white] [circle,inner sep=0mm, minimum size=1.7mm] {} 
		   	edge from parent[thin] node [below] {{\footnotesize$e_2$}}
			}
		   }
		edge from parent[line width=1mm]
		}
	    	child{node[thin,fill=white] [circle,inner sep=0mm, minimum size=1.7mm] {} edge from parent 
		edge from parent[thick]
		}
		child {node [thin,fill=white] [circle,inner sep=0mm, minimum size=1.7mm] {} 
		   child[rotate=60]{node [thin,fill=black] [circle,inner sep=0mm, minimum size=1.7mm] {}
			child {node [thin,fill=white] [circle,inner sep=0mm, minimum size=1.7mm] {} 
		   	edge from parent[thin] node [below right] {{\footnotesize$\delsig e_6$}}
			}
			child {node [thin,fill=white] [circle,inner sep=0mm, minimum size=1.7mm] {} 
		   	edge from parent[thin] node [left] {{\footnotesize$e_6$}}
			}
		   }
		   child[rotate=60]{node [thin,fill=black] [circle,inner sep=0mm, minimum size=1.7mm] {}
			child {node [thin,fill=white] [circle,inner sep=0mm, minimum size=1.7mm] {} 
		   	edge from parent[thin] node [right] {{\footnotesize$\delsig e_5$}}
			}
			child {node [thin,fill=white] [circle,inner sep=0mm, minimum size=1.7mm] {} 
		   	edge from parent[thin] node [left] {{\footnotesize$e_5$}}
			}
		   }
		edge from parent[line width=1mm]
		}
		child {node [thin,fill=white] [circle,inner sep=0mm, minimum size=1.7mm] {} 
		   child{node [thin,fill=black] [circle,inner sep=0mm, minimum size=1.7mm] {}
		   edge from parent[thin] node [below] {{\footnotesize$\delsig e_7$}}
		   }
		   child{node [thin,fill=black] [circle,inner sep=0mm, minimum size=1.7mm] {}
		   edge from parent[thin] node [above] {{\footnotesize$e_7$}}
		   }
		edge from parent[line width=1mm]
		};
	    \node[grow=up] at (0,5) [thin,fill=black] [circle,inner sep=0mm, minimum size=1.7mm] {} 
	    		    child{node[thin,fill=white] [circle,inner sep=0mm, minimum size=1.7mm] {} 
	    			child[rotate=90] {node[thin,fill=black] [circle,inner sep=0mm, minimum size=1.7mm] {}
					edge from parent [thin] node [right] {{\footnotesize$\delsig e_4$}}}
				child[rotate=90] {node[thin,fill=black] [circle,inner sep=0mm, minimum size=1.7mm] {}
					edge from parent [thin] node [left] {{\footnotesize$e_4$}}}
			    edge from parent[line width=1mm] {}
			};
	\node at (0,4.0) {$\vdots$};
\end{tikzpicture}
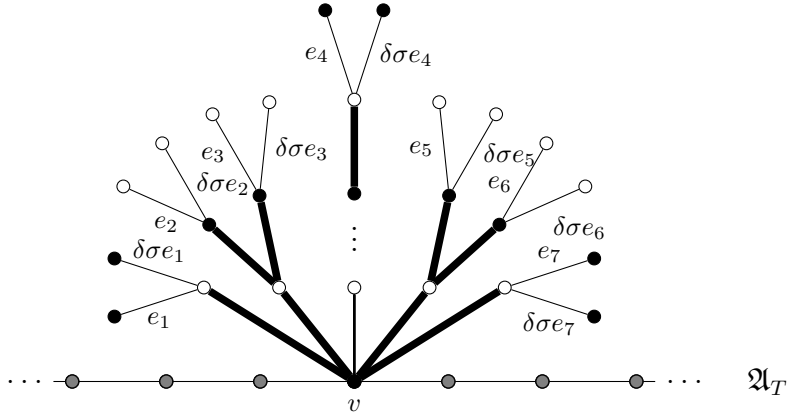
\captionof{figure}[The action of $\delsig$ on $\build(H_E)$, $F'/F$ unramified]{Edges distance one from $\build(H_E)^{\delsig}$. Here $\build(H_E)^{\delsig}$ is in bold. Note that the type of the farthest vertex of an edge depends on the distance from the $\delsig$-fixed point nearest that edge to $v$.}
\end{center}

Now consider those edges whose nearest $\delsig$-fixed point is distance $0<d<a$ from $v$. 
We know from the second paragraph of this proof that there are exactly $(q+1)q^{d-1}$ vertices of $\build(H_E)^{\delsig}$ distance $d$ from $v$, and, again, lemma \ref{langlemmatwo} shows that for each such vertex $v'$ there are $q^{fr}(1-q^{1-f})$ edges at distance $r$ from $\build(H_{E})^{\delsig}$ and whose nearest $\delsig$-fixed point is $v'$. This makes for a total of $(q+1)(1-q^{1-f})q^{fr}q^{d-1}$ edges distance $r$ from $\build(H_E)^{\delsig}$ whose nearest $\delsig$-fixed point is distance $d$ from $v$. By construction  $k_e=r+d$ for each such edge $e$, and so all of such edges lie in $X^0_{w_1}(\delsig)$ if $r+d$ is odd and $X^0_{w_2}(\delsig)$ if $r+d$ is even. 

There are no $\delsig$-fixed edges distance farther than $a$, so for any $\delsig$-fixed vertex $v'$ distance $a$ from $v$, the number of edges distance $r$ from $\build(H_E)^{\delsig}$ whose nearest $\delsig$-fixed point is $v'$ is $q^{fr}$. Lemma \ref{langlemmatwo} tell us that there are $(q+1)q^a$ such vertices $v'$, and so there are a total of  $(q+1)q^{a-1}q^{fr}$ edges distance $r$ from $\build(H_E)^{\delsig}$ whose nearest $\delsig$-fixed point is distance $a$ from $v$. As before, all such edges lie in $X_{w_1}^0(\delsig)$ when $a+r$ is odd and in $X^0_{w_2}(\delsig)$ when $a+r$ is even.

In order to actually determine the cardinalities of $X^0_{w_1}(\delsig)$ and $X_{w_2}^0(\delsig)$, we will need to break into cases. Specifically, for a fixed $r\geq 1$, we must determine the cardinalities of these sets in the case when both $r$ and $a$ are odd; $r$ is odd and $a$ is even; $r$ is even and $a$ is odd; and both $r$ and $a$ are even. We will be assuming that $v$ has type zero until explicitly stated otherwise.

\subparagraph*{\textbf{$r$ and $a$ odd.}}
When $r$ is odd, $X_{w_1}^0(\delsig)$ is comprised of those edges whose nearest $\delsig$-fixed point is an even distance from $v$, while $X_{w_2}^0(\delsig)$ is comprised of those edges whose nearest $\delsig$-fixed point is an  odd distance from $v$. Suppose that $a$ is also odd. Then for $X^0_{w_1}(\delsig)$ we add in all the edges corresponding to $d=0$, $2$, $\cdots$, $a-1$, so that
\begin{align*}
\# X^0_{w_1} &= q^{fr}(1-q^{1-f})\left[1+(q+1)\sum_{i=1}^{\frac{a-1}{2}} q^{2i-1}\right]\\
	&= q^{fr}(1-q^{1-f})\left[1+(q+1)q\sum_{i=1}^{\frac{a-1}{2}} q^{2i-2}\right]\\
	&= q^{fr}(1-q^{1-f})\left[1+(q+1)q\sum_{i=0}^{\frac{a-1}{2}-1} q^{2i}\right]\\
	&= q^{fr}(1-q^{1-f})\left[1+(q+1)q\frac{q^{a-1}-1}{q^2-1}\right]\\
	&= q^{fr}(1-q^{1-f})\left[1+q\frac{q^{a-1}-1}{q-1}\right]\\
	&= q^{fr}(1-q^{1-f})\frac{q^a-1}{q-1},
\end{align*}
and for $X^0_{w_2}(\delsig)$ we add in all those edges corresponding to $d=1$, $3$, $\cdots$, $a-2$, and $a$, giving
\begin{align*}
\# X^0_{w_2} &= q^{fr}(1-q^{1-f})\left[(q+1)\sum_{i=1}^{\frac{a-1}{2}} q^{(2i-1)-1}\right]+q^{fr}(q+1)q^{a-1}\\
	&= q^{fr}(1-q^{1-f})\left[(q+1)\sum_{i=0}^{\frac{a-1}{2}-1} q^{2i}\right]+q^{fr}(q+1)q^{a-1}\\
	&= q^{fr}(1-q^{1-f})\left[(q+1)\frac{q^{a-1}-1}{q^2-1}\right]+q^{fr}(q+1)q^{a-1}\\
	&= q^{fr}\left[(1-q^{1-f})\frac{q^{a-1}-1}{q-1}+(q+1)q^{a-1}\right].
\end{align*}

\subparagraph*{\textbf{$r$ odd, $a$ even.}}Now suppose that $a$ is even. Then for $X^0_{w_1}(\delsig)$ we add in all the edges corresponding to $d=0$, $2$, $\cdots$, $a-2$, and $a$, which yields
\begin{align*}
\# X^0_{w_1} &= q^{fr}(1-q^{1-f})\left[1+(q+1)\sum_{i=1}^{\frac{a-2}{2}} q^{2i-1}\right]+q^{fr}(q+1)q^{a-1}\\
	&= q^{fr}(1-q^{1-f})\left[1+(q+1)q\sum_{i=1}^{\frac{a-2}{2}} q^{2i-2}\right]+q^{fr}(q+1)q^{a-1}\\
	&= q^{fr}(1-q^{1-f})\left[1+(q+1)q\sum_{i=0}^{\frac{a-2}{2}-1} q^{2i}\right]+q^{fr}(q+1)q^{a-1}\\
	&= q^{fr}(1-q^{1-f})\left[1+(q+1)q\frac{q^{a-2}-1}{q^2-1}\right]+q^{fr}(q+1)q^{a-1}\\
	&= q^{fr}(1-q^{1-f})\left[1+q\frac{q^{a-2}-1}{q-1}\right]+q^{fr}(q+1)q^{a-1}\\
	&= q^{fr}\left[(1-q^{1-f})\left(\frac{q^{a-1}-1}{q-1}\right)+(q+1)q^{a-1}\right],
\end{align*}
and for $X^0_{w_2}(\delsig)$ we add in all those edges corresponding to $d=1$, $3$, $\cdots$, $a-1$, resulting in
\begin{align*}
\# X^0_{w_2} &= q^{fr}(1-q^{1-f})\left[(q+1)\sum_{i=1}^{\frac{a}{2}} q^{(2i-1)-1}\right]\\
	&= q^{fr}(1-q^{1-f})\left[(q+1)\sum_{i=0}^{\frac{a}{2}-1} q^{2i}\right]\\
	&= q^{fr}(1-q^{1-f})\left[(q+1)\frac{q^{a}-1}{q^2-1}\right]\\
	&= q^{fr}(1-q^{1-f})\frac{q^{a}-1}{q-1}.
\end{align*}

\subparagraph*{\textbf{$r$ even, $a$ even or odd.}}When $r$ is even, we get almost exactly the same as the above; the only difference is that the edges in $X^0_{w_1}(\delsig)$ are those with $d$ odd while those in $X^0_{w_2}(\delsig)$ are those with $d$ even; that is, we get the same formulae as for the case of odd $r$, with the roles of $X^0_{w_1}(\delsig)$ and $X^0_{w_2}(\delsig)$ switched, as desired.

Two things remain to be proven: first, that the action of $\delsig$ fixes precisely one point, $v$, of $\build(H_{E})$ that also lies in $\apt$, and that $v$ is a vertex; and second, that this vertex is type zero except when $\gamma$ splits in $E$ and the difference of the $E$-valuations of the eigenvalues of $\delta$ is not divisible by $4$. The second fact will be proven in the course of investigating the first. Note that we are no longer assuming $v$ is type zero; instead we are determining what types it can have, and when.

 Consider first the case where $\gamma$ does not split in $E$. Then $\delta$ also does not split in $E$, since it is an element of the same elliptic torus as $\gamma$. We analyze the action of $\delsig$ by studying $\build(H_{E})$ as a subtree of $\build(H_{EF'})$. 
Section \ref{lunram} shows that there is exactly one $\gal(EF',E)$-fixed point in $\apt(EF')$ and that it is $v_0$. Since $\delta\in T_E$, it sends $\apt(EF')$ to itself and $E$-points to themselves, $\delta$ must also fix $v_0$, which has type zero.

Finally, consider the case where $\gamma$ does split in $E$ and $s$ is even. If the eigenvalues of $\delta$ have valuations $m$ and $n$, then $s=m+n=2k$, while $\delta$ shifts $\apt(E)$ by $m-n=2k-2n$. We know that $\sigma$ acts on $\apt(E)$ by the action $x\rightarrow -x$. Thus $\delsig(x)=-x+m-n$ for all points $x\in\apt(E)$. Solving for the fixed points of this map yields $x=\frac{m-n}{2}=k-n$, an integer, and therefore a vertex of $\apt(E)$. Since the vertex corresponding to zero, $v_0$, is type zero, the vertex at $x$ is type zero when $x$ is even and type $1$ when $x$ is odd. But $x$ is even if and only if $m-n$ is divisible by 4.
\end{proof}



\section{The Iwahori-Hecke Algebra}\label{heckechapt}

Recall that $\mathcal{H}_L$ is the Iwahori-Hecke algebra, the set of compactly-supported $I$-bi-invariant functions from $GL_2(L)$ to $\C$. The integrands in our main theorem lie in $\zhi$ and $\zhie$, the centers with respect to convolution of $\mathcal{H}$ and $\mathcal{H}_E$.  We need a description of these algebras and of the base change homomorphism between them.
\subsection{The $\mu$-Admissible Set}\label{admiss}
In order to describe the basis of \zhie\ that we will use, we must address the notion of $\mu$-admissibility of elements of the extended affine Weyl group. Let $\mu\in X_*(S_L)\isom\Z^2$. The finite Weyl group $W$ acts on $X_*(S_L)$ as permutations on $\Z^2$.  An element $w\in\wae$ is \emph{$ \mu $-admissible} if there exists a $w_0\in W$ such that  $w\leq t_{w_0\mu}$ in the Bruhat order on $\wae$. 

Unraveling this definition, we see that $w=(m,b,s)$ is $\mu=(i,j)$-admissible if $s=-i-j$ and $\ell(m,b)\leq |i-j|$. We call $-i-j$ and $|i-j|$ the \emph{size} and \emph{length} of $\mu$, respectively; this is consistent with our use of the same words for $t_\mu$. Using our identification of $\wae$ with the extended edges of $\sapt'(L)$, this means that the  $e'\in\sapt'(L)$ corresponding to $w$ is $\mu$-admissible if and only if $\size(e)=\size(\mu)$ and the length of the minimal gallery between $e_0$ and $e$ is less than or equal to $\ell(\mu)$, where $e$ is the image of $e'$ in $\sapt(L)$. Below are pictures of the $\mu$-admissible sets for $\mu=(i,i)$ and $\mu=(0,-1)$; in each case the admissible edges are bold.
\begin{center}
\begin{tikzpicture}
[axis/.style={<->,>=stealth'},every circle node/.style={draw,circle,
			inner sep=0mm},scale=.6,grow cyclic,
	level 1/.style={level distance=8mm,sibling angle=180},
	level 2/.style={level distance=6mm,sibling angle=68},
	level 3/.style={level distance=6mm,sibling angle=57}]
   	\draw[white,myBG] (-9.25,0) -- (9.25,0);
	\draw[black,line width=1mm] (-1.25,0)--(1.25,0);
   	\draw[black] (-9.25,0) -- (9.25,0);
   	\draw (-9.25,0) node[left] {$\ldots$}; 
   	\draw (9.25,0) node[right] {$\ldots$};
   	\node at (12,0) {$\sapt(L')\tau^{-2i}$};  			
    	\foreach \colorA/\colorB/\shift in {black/white/0,white/black/2.5}
   	  \foreach \pos in {-6.25-\shift,-1.25-\shift,3.75-\shift,8.75-\shift}{
 	    \node[fill=white] at (\pos,0) [circle, minimum size=1.8mm] {};
 	    \node[fill=\colorA] at (\pos,0) [circle, minimum size=1.7mm] {}; 
 	}
 	\node[fill=white] at (-8.75,0) [circle, minimum size=1.8mm] {};
 	    \node[fill=white] at (-8.75,0) [circle, minimum size=1.7mm] {};
 	    
 	  \foreach \pos/\postext in {-7.5/{(1,1,-2i)},-5/{(1,0,-2i)},-2.5/{(0,1,-2i)},-0/{(0,0,-2i)},2.5/{(-1,1,-2i)},5/{(-1,0,-2i)},7.5/{(-2,1,-2i)}} 
 	  \node at (\pos,0) [below] {\tiny$\postext$};

  \begin{scope}[yshift=-3cm]
   	\draw[white,myBG] (-9.25,0) -- (9.25,0);
	\draw[black,line width=1mm] (-3.75,0)--(3.75,0);
   	\draw[black] (-9.25,0) -- (9.25,0);
   	\draw (-9.25,0) node[left] {$\ldots$}; 
   	\draw (9.25,0) node[right] {$\ldots$};
   	\node at (12,0) {$\sapt(L')\tau$};  			
    	\foreach \colorA/\colorB/\shift in {black/white/0,white/black/2.5}
   	  \foreach \pos in {-6.25-\shift,-1.25-\shift,3.75-\shift,8.75-\shift}{
 	    \node[fill=white] at (\pos,0) [circle, minimum size=1.8mm] {};
 	    \node[fill=\colorA] at (\pos,0) [circle, minimum size=1.7mm] {}; 
 	}
 	\node[fill=white] at (-8.75,0) [circle, minimum size=1.8mm] {};
 	    \node[fill=white] at (-8.75,0) [circle, minimum size=1.7mm] {};
 	    
 	  \foreach \pos/\postext in {-7.5/{(1,1,1)},-5/{(1,0,1)},-2.5/{(0,1,1)},-0/{(0,0,1)},2.5/{(-1,1,1)},5/{(-1,0,1)},7.5/{(-2,1,1)}} 
 	  \node at (\pos,0) [below] {\tiny$\postext$};

  \end{scope}
\end{tikzpicture}
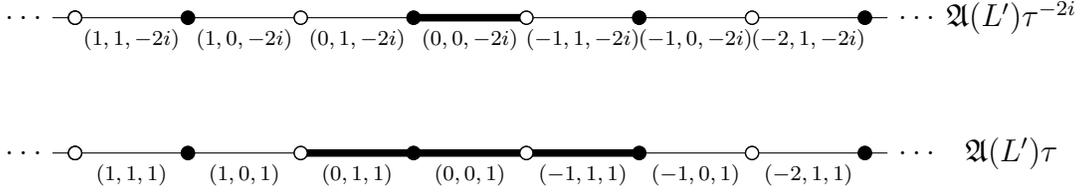
\captionof{figure}[The $\mu$-admissible sets, $\mu=(i,i)$ and $\mu=(0,-1)$.]{\emph{Top}: The $(i,i)$-admissible set. \emph{Bottom}: The $(0,-1)$-admissible set.}
\end{center}

\subsection{Bernstein Functions and Base Change}\label{bernsec}
The Bernstein isomorphism is an isomorphism $B_L:\C[X_*(S_L)]^{W}\isomto \zhil$. Denote by $g_\lambda$ the element $\lambda\in X_*(S_L)$ viewed as an element of $\C[X_*(S_L)]$. Then, following \cite{lusztig}, we define the \emph{Bernstein function} $z_\mu\in \zhil$ by
$$z_\mu=B_L\left(\sum_{\lambda\in W\mu}g_\lambda\right).$$ 
\begin{theorem}\label{berncoefs}
Let $\mu\in X_*(S_L)$ be dominant. Then $z_\mu=\sum_{w\in\wae} z_\mu(w)1_{I_LwI_L}$, where
$$z_\mu(x)=\begin{cases} 0, & x\text{ not }\mu\text{-admissible}\\
					q^\frac{-\ell(\mu)}{2}, & \ell(\mu)-\ell(x)=0\\
					q^\frac{-\ell(\mu)}{2}((-q)^r-1)\left(\frac{q-1}{q+1}\right), & \ell(\mu)-\ell(x)=r\geq1,\end{cases}$$
where $q^n$ is the cardinality of the residue field of $L$. The $z_\mu$ form a $\C$-basis of $\zhil$.
\end{theorem}
\begin{proof}
That the Bernstein functions form a $\C$-basis of \zhil\ is proven in \cite{lusztig}. The coefficients of the functions were worked out in \cite{hainesbern}, corollary $10.4$ to be 
$$z_\mu(x)=\left\{\begin{array}{l l} 0, & x\text{ not }\mu\text{-admissible}\\
					q^\frac{-\ell(\mu)}{2}, & \ell(x)=\ell(\mu)\\
					q^\frac{-\ell(\mu)}{2}(1-q), & \ell(x)=\ell(\mu)-1\\
					q^\frac{-\ell(\mu)}{2}(1-2q+2q^2-\cdots+2(-q)^{r-1}+(-q)^r), & r=\ell(\mu)-\ell(x)\geq2,
		\end{array}\right.$$
Since $$1-q=((-q)^1-1)\frac{q-1}{q+1}$$ is obvious and $$(1-2q+2q^2-\cdots+2(-q)^{r-1}+(-q)^r)=((-q)^r-1)\frac{q-1}{q+1}$$ follows from factoring out $q-1$ and then simplifying the geometric series, the theorem follows.
\end{proof}
Since the $z_\mu(x)$ depend only upon $s=\ell(\mu)-\ell(x)$, it will be convenient to write this coefficient as $z_\mu^s$.

Let $L$ be unramified over $F$. The \emph{base change homomorphism for Iwahori-Hecke algebras} is defined in \cite[3.2]{haineslemma} as the function $b$ which makes the following diagram commute:
\begin{center}
\begin{tikzpicture}[description/.style={fill=white,inner sep=2pt},
bij/.style={below,sloped,inner sep=2.5pt}]
\matrix (m) [matrix of math nodes, row sep=3em,
column sep=2.5em, text height=1.5ex, text depth=0.25ex]
{ \C[X_*(S_L)]^W & \zhil \\
 \C[X_*(S)]^W &  \zhi\\ };
\path[->,font=\scriptsize]
(m-1-1) edge node[bij] {$\sim$} node[above] {$ B_L $} (m-1-2)
edge node[left] {$ N $} (m-2-1)
(m-2-1) edge node[bij] {$\sim$}  node[above] {$ B $} (m-2-2)
(m-1-2) edge node[auto] {$ b $} (m-2-2);
\end{tikzpicture}\captionme{The diagram defining the base change homomorphism $b$}\label{basechangedef}
\end{center}
Here $N$ is the norm map defined by $N(g_\mu)=\sum_{i=0}^{[L:F]-1} \sigma_L^ig_\mu$ for any $\mu\in X_*(S_L)$. 

\begin{theorem}\label{basechangeforbernstein}
Let $[L:F]=r$. Then $b(z_{\mu})=z_{r\mu}$.
\end{theorem}
This is not a new result, but we provide a proof here for completeness.
\begin{proof}
Since $\mathbf{G}$ is split, the norm map is simply $N(g_\mu)=rg_\mu=g_{r\mu}$, where $c\mu=(cm,cn)$ if $\mu=(m,n)$. Then 
\begin{align*}b(z_\mu) &= B\left(N\left(\sum_{\lambda\in W\mu} g_\lambda\right)\right)\\
&= B\left(\sum_{\lambda\in W\mu} g_{r\lambda}\right)\\
&= B\left(\sum_{\lambda\in W(r\mu)} g_{\lambda}\right)\\
&= z_{r\mu}
\end{align*}
\end{proof}

In light of the results in subsection \ref{admiss}, we may give another description of $z_\mu$ by breaking up the sum into the parts involving $w\in\wae$ of length zero, odd length, and non-zero even length. This simple rearrangement will be very useful later.
\begin{corollary}\label{berncor}
Let $\siz(\mu)=s$. We have 
\begin{align*}z_\mu &= z_\mu^{\ell(\mu)}1_{(0,0,s)}\\
	&\qquad +\sum_{1\leq i\leq \frac{\ell(\mu)+1}{2}} z_\mu^{\ell(\mu)-2i+1}\left[1_{(-i,1,s)}+1_{(i-1,1,s)}\right]\text{ (odd length edges)}\\
	&\qquad +\sum_{1\leq j\leq\frac{\ell(\mu)}{2}} z_\mu^{\ell(\mu)-2j}\left[1_{(-j,0,s)}+1_{(j,0,s)}\right]\text{ (even length edges)}.
\end{align*}
\end{corollary}
\begin{proof}
The sizes are as determined in subsection \ref{admiss}. The first term in the summation belongs to the unique length zero edge, the first summation covers the $\mu$-admissible edges of odd length, while the second covers those of non-zero even length.
\end{proof}




\section{Computation of the Orbital Integrals}\label{orbitalchapt}
From now until section \ref{fltchapt} we choose a $\delta\in G_E$ such that $N(\delta)=\gamma$ is semi-simple regular elliptic and we normalize $\meas_{dg_E}(I_E)=1$.
\subsection{Langlands' Results on the Spherical Hecke Algebra}\label{oldlanglands}
The reader may find it interesting to compare the orbital and twisted orbital integrals of characteristic functions in the Iwahori-Hecke algebra obtained here with those obtained by Langlands for characteristic functions in the spherical Hecke algebra. Since the results are spread out across the whole of section 5 of \cite{langlands}, I collect them here. 

Recall that $K_L$ is $G_{\ints_L}$ and $T_{\ints}=T\cap K=T\cap K_E$.
\begin{theorem}
Let $[E:F]=f$ be prime, let $\delta\in G_E$ have norm conjugate to a semi-simple regular elliptic $\gamma\in G$ with $\Delta(\gamma)=q^{-r}$, 
let $M_L=\frac{\meas_{dg_T}(T_{\ints})}{\meas_{dg}(K_L)}$, 
 let $\mu=(i,j)$ be a dominant cocharacter, and let 
$$a_L=\begin{pmatrix} \pi_L^i & 0 \\0 & \pi_L^j\end{pmatrix},$$
$L=F$ or $E$.
 Then the orbital and twisted orbital integrals of characteristic functions $1_{Ka_FK}$ and $1_{(K_Ea_EK_E)}$, respectively, are all zero except the following cases:
\begin{enumerate}
\item When $F'$ is unramified over $F$ and $n$ is the common valuation over $F'$ of the eigenvalues of $\gamma$, then
\begin{enumerate}
\item the orbital integral is
$$
M_F\orb(1_{Ka_FK})=\begin{cases}
			\frac{q^{r}(q+1)-2}{q-1}, & (i,j)=(n,n)\\
			q^{r+m}(q+1), & (i,j)=(n+m,n-m), m> 0
			\end{cases}
$$
\item if $T$ splits over $E$ (i.e. if $f=2$) and $\val_E\det(\delta)=2n+1$ is odd, then
$$
M_E\orb(1_{Ka_EK})= 2q^{fr}, (i,j)=(n+1+m,n-m),m\geq 0
$$
\item if $T$ splits over $E$ and $\val_E\det(\delta)=2n$ is even, or if $T$ does not split over $E$, then
$$
M_E\orb(1_{Ka_EK})=\begin{cases}
			\frac{q^{fm}}{q-1}\left[q^{a}(q+1)(1-q^{-f})-2(1-q^{1-f})\right], & ~\\
				\qquad (i,j)=(n+m,n-m),m>0\text{ if }T\text{ splits}& ~\\
				\qquad\quad \text{over  }E\\
				\qquad\qquad\qquad\qquad \text{or}\\
				\qquad (i,j)=(\frac{v}{f}+m,\frac{v}{f}-m), m>0\text{ if }T\text{ does not}\\
				\qquad\quad \text{split over  }E\\
			\frac{q^{a}(q+1)-2}{q-1},  (i,j)=(n,n)\text{ if }T\text{ splits over }E & ~\\
			\qquad \text{or }(i,j)=(\frac{v}{f},\frac{v}{f})\text{ if }T\text{ does not split over }E
			\end{cases}
$$
\end{enumerate}
\item When $F'$ ramified over $F$, then 
\begin{enumerate}
\item if $\vd(\gamma)=2n+1$ is odd, then 
$$
M_F\orb(1_{Ka_FK})= q^m, (i,j)=(n+1+m,n-m), m\geq0
$$
\item if $\vd(\gamma)=2n$ is even, then, letting $r=a+\frac{d_T+1}{2}$, we get
$$
M\orb(1_{Ka_FK})=\begin{cases}
			\frac{q^a-1}{q-1}, & (i,j)=(n,n)\\
			q^{a+m}, & (i,j)=(n+m,n-m), m>0\end{cases}
$$
%
\item if $\vd(\delta)=2n+1$ is odd, then
$$M_E\twisted(1_{Ka_EK})=q^{fm}, (i,j)=(n+1+m,n-m), m\geq 0$$
\item if $\vd(\delta)=2n$ is even, then, letting $r=a+\frac{d_T+1}{2}$, we get
$$M_E\twisted(1_{Ka_EK})=\begin{cases}
			q^{fm}\left[(1-q^{1-f})\frac{q^a-1}{q-1}+q^a\right], & \quad\\
				\qquad (i,j)=(n+m,n-m),m>0\\
			\frac{q^{a+1}-1}{q-1},  (i,j)=(n,n)
			\end{cases}
$$
\end{enumerate}
\end{enumerate}\qed
\end{theorem}

\subsection{Changing from Integration to Counting: Common Preliminaries}\label{sumsec}
Two lemmas important for changing the orbital and twisted orbital integrals into counting problems in the building for $GL_2$ do not depend upon whether or not $F'$ is ramified over $F$. These results are stated for the field $E$, which includes the trivial extension $E=F$ where $\sigma$ is the identity map.
\begin{lemma}\label{reductionone}The twisted orbital integral of a characteristic function in the Iwahori-Hecke algebra is given by
\begin{align*}
\twisted(1_{I_EwI_E}) &= \sum_{[e']\in T\backslash X_{w}(\delta\sigma)}\left(\frac{1}{\meas_{dg_{T}}(T\cap g_{e'}I_Eg_{e'}^{-1})}\right)
\end{align*}
for any $w\in\wae$, where $[e']$ denotes the class $Te'$ and $g_{e'}\in G_E$ is defined by $e'=g_{e'}e_0'$. 
\end{lemma}
\begin{proof}
%
The equality follows immediately upon applying the same reasoning used in the proof of \cite{kottwitz} $(3.4.1)$ to the twisted orbital integral, using lemma \ref{toruslemma} to replace $G_\delta$ with $T$ throughout.
\end{proof}

\begin{lemma}\label{oneofflemma}
Let $i\in\Z$. Then $$\sum_{[e']\in T_{\ints}\backslash X_w^i(\delta\sigma)} \frac{1}{\meas_{dg_T}(T_{\ints}\cap g_{e'}I_Eg_{e'}^{-1})}=\frac{\# X_w^i(\delta\sigma)}{\meas_{dg_T}(T_{\ints})}.$$
\end{lemma}
\begin{proof}
Since $g_{e'}Ig_{e'}^{-1}$ is the fixer in $G_E$ of $e'$, the size of the fiber of $[e']$ under the morphism  $X_w^i(\delta\sigma)\rightarrow T_{\ints}\backslash X_w^I(\delsig)$ is equal to the index in $T_{\ints}$ of $T_{\ints} \cap g_{e'}Ig_{e'}^{-1}$. The group $g_{e'}Ig_{e'}^{-1}$ is a compact open subgroup of the compact group $T_{\ints}$, and so this index is finite. We have 
$$\sum_{e'\in[e']}\frac{1}{\meas_{dg_T}(T_{\ints}\cap g_{e'}I_Eg_{e'}^{-1})}=\frac{[T_{\ints}:T_{\ints}\cap g_{e'}Ig_{e'}^{-1}]}{\meas_{dg_T}(T_{\ints}\cap g_{e'}I_Eg_{e'}^{-1})};$$ rearranging the terms yields
$$\sum_{e'\in [e']} \frac{1}{\meas_{dg_T}(T_{\ints}\cap g_{e'}I_Eg_{e'}^{-1})}\frac{1}{[T_{\ints}:T_{\ints}\cap g_{e'}Ig_{e'}^{-1}]}=\frac{1}{\meas_{dg_T}(T_{\ints}\cap g_{e'}I_Eg_{e'}^{-1})}.$$
The product of the denominators on the left hand side is exactly $\meas_{dg_T}(T_{\ints})$. Upon substituting we get
\begin{align*}
\sum_{[e']\in T_{\ints}\backslash X_w^i(\delta\sigma)} \frac{1}{\meas_{dg_T}(T_{\ints}\cap g_{e'}I_Eg_{e'}^{-1})} &= \sum_{[e']\in T_{\ints}\backslash X_w^i(\delta\sigma)}\sum_{e'\in [e']} \frac{1}{\meas_{dg_T}(T_{\ints})}\\
&= \frac{1}{\meas_{dg_T}(T_{\ints})}\sum_{e'\in X_w^i(\delta\sigma)} 1\\
&= \frac{\# X_w^i(\delta\sigma)}{\meas_{dg_T}(T_{\ints})}.
\end{align*}
\end{proof}

\subsection{Computing the Orbital Integrals}\label{lemmasec}
Here we compute the orbital integrals. The last steps in the transition from integration to counting depend upon the structure of $T$, and therefore ultimately depend on the ramification or not of $F'$ over $F$.
\subsubsection{$F'$ unramified over $F$}
\begin{lemma}\label{unramdecomp}Assume $T$ is in standard form. Then the group $T$ admits the decomposition $T=\left\langle t_{(1,1)}\right\rangle\times T_{\ints}$.
\end{lemma}
\begin{proof}
The action of $T$ on $\build(G_{F'})$ preserves $\apt'(F')$ and $\build(G)$. The only $F$-points of $\apt'(F')$ are those lying above $v_0$. Since $T\subset G$, its action must preserve the set of points lying above $v_0$. 
The stabilizer of this set is  $Z(G_{F'})K_{F'}$ by definition, so we obtain $T\subset Z(G)K$. But $Z(G)$ decomposes as $\left\langle t_{(1,1)}\right\rangle\times\left(Z(G)\cap K\right)$, and so $T=\left\langle t_{(1,1)}\right\rangle\times T_{\ints}$.
%
%
\end{proof}

\begin{corollary}\label{unramreduction}
If $T$ is in standard form, the twisted orbital integral of a characteristic function is given by
$$\twisted(1_{I_EwI_W})=\frac{1}{\meas_{dg_T}(T_{\ints})}\left[\# X_w^0(\delta\sigma) + \# X_{\overline{w}}^0(\delta\sigma)\right].$$
\end{corollary}
\begin{proof}
Lemma \ref{unramdecomp} tells us that $T\backslash X_w(\delta\sigma)=\left\langle t_{(1,1)}\right\rangle \times T_{\ints}\backslash X_w(\delta\sigma)$. The action of $t_{(1,1)}^n$ on $X_w(\delta\sigma)$ restricts to a bijection of $X_w^i(\delta\sigma)$ onto $X_w^{i-2n}(\delta\sigma)$, and therefore $T\backslash X_w(\delta\sigma)=T_{\ints}\backslash (X_w^0(\delta\sigma)\coprod X_w^1(\delta\sigma))$. This and lemmas \ref{reductionone} and \ref{oneofflemma} imply that 
$$\twisted(1_{I_EwI_W})=\frac{1}{\meas_{dg_T}(T_{\ints})}\left[\# X_w^0(\delta\sigma) + \# X_{w}^1(\delta\sigma)\right].$$ The result follows since $\#X_{w}^1(\delta\sigma)=\#X_{\overline{w}}^0(\delta\sigma)$ (recall \ref{sizelemma}).
\end{proof}
\begin{theorem}\label{flunramorbit}
Normalize $dg_T$ so that $\meas_{dg_T}(T_{\ints})=1$. Fix a semi-simple regular elliptic $\gamma\in G$ such that $F'/F$ is unramified, let $\Delta(\gamma)=q^{-a}$, and let $\mu \in X_*(S)$. Then 
$$\orb(z_{\mu	})=\begin{cases} 
			0, & \val\det(\gamma)\neq \size(\mu)\\
			2(q^a-1)\left(\frac{q+1}{q-1}\right), & \ell(\mu)=0\\
			2q^{-\ell(\mu)/2}(1-(-q)^{\ell(\mu)}), & \ell(\mu)\neq 0\end{cases}
$$
\end{theorem}
\begin{proof}
Since $\orb$ depends only on the conjugacy class of $\gamma$, we may assume without loss of generality that $\gamma$ is in standard form.
The proof rests on lemma \ref{unramcount} and corollary \ref{berncor}; the reader may find it useful to refer back them at this point. Together they imply that the integral will be zero for all $\gamma$ with size not equal to $\size(\mu)$, so suppose that $s=\size(\gamma)=\size(\mu)$.

When $\ell(\mu)=0$, \ref{unramcount} and \ref{berncor} tell us immediately that 
$$\orb(z_\mu)=2(q^a-1)\left(\frac{q+1}{q-1}\right).$$ 

Suppose that $\ell(\mu)\neq 0$.
Comparing the expansion of $z_{\mu}$ in \ref{berncor} with the list of non-empty $\xwgz$ in \ref{unramcount} we see that, other than the length zero term, the even length $\mu$-admissible $w$ contribute nothing to the sum (these are the $1_{IwI}$ with $b=0$, and all such \xwgz\ except the zero-length term are empty). By corollary \ref{unramreduction} the integral is 
\begin{align*}\orb(z_{\mu}) &= z_{\mu}^{\ell(\mu)}\left[\#X^0_{(0,0,s)}(\gamma)+\#X^0_{(0,0,s)}(\gamma)\right]\\
  		&\qquad + \sum_{1\leq i\leq \frac{\ell(\mu)+1}{2}} 
 			z_{\mu}^{\ell(\mu)-2i+1} \left[\left(
 			\#X^0_{(-i,1,s)}(\gamma)+\#X^0_{\overline{(-i,1,s)}}(\gamma)
 			\right)\right.\\
		&\qquad\qquad\qquad\qquad\qquad\qquad \left.+ \left(\#X^0_{(i-1,1,s)}(\gamma)+
 			\#X^0_{\overline{(i-1,1,s)}}(\gamma)
 			\right)\right].\\
\intertext{Since $\overline{(-i,1,s)}=(i-1,1,s)$ and vice-versa for all $i\geq 1$, we get}
 \orb(z_{\mu})	&= 2\left[z_{\mu}^{\ell(\mu)}\left(\#X^0_{(0,0,s)}(\gamma)\right)\right.\\
		&\qquad+\sum_{1\leq i\leq \frac{\ell(\mu)+1}{2}} 
			\left.z_{\mu}^{\ell(\mu)-2i+1} \left(
			\#X^0_{(-i,1,s)}(\gamma) + \#X^0_{(i-1,1,s)}(\gamma)\right)\right].
\end{align*}

  A further simplification of the sum occurs when we pay attention to the issues of parity or disparity between $a$ and $i$, which varies as we go through the sum of the odd length edges.
 Indeed, between the pair $\#X^0_{(-i,1,s)}(\gamma)$ and $\#X^0_{(i-1,1,s)}(\gamma)$ which appear in each iteration of the summation, only one will be non-zero; either $i$ and $a$ share parity, in which case $X^0_{(i-1,1,s)}(\gamma)$ is empty, or $i$ and $a$ do not share parity, in which case $X^0_{(-i,1,s)}(\gamma)$ is empty. Thus $(\#X^0_{(-i,1,s)}(\gamma)+\#X^0_{(i-1,1,s)}(\gamma))=(q+1)q^{a+i-1}$ for all $i=1,\ldots,\frac{\ell(\mu)+1}{2}$.

We must, because of the peculiarities of the coefficients and sums involved, handle the case of $\ell(\mu)$ even separately from the case of $\ell(\mu)$ odd. We pursue the odd case first; let $\ell(\mu)=2m-1$. The above and substituting for the coefficients in theorem \ref{berncoefs} result in  
\begin{align*}
 \orb(z_{\mu}) &= 2q^{-\frac{\ell(\mu)}{2}}\left[((-q)^{2m-1}-1)\frac{q-1}{q+1} (q^a-1)\frac{q+1}{q-1}\right.\\
		&\qquad\qquad+\left.\sum_{1\leq i\leq m-1} ((-q)^{2m-2i}-1)\frac{q-1}{q+1} \bigg((q+1)q^{a+i-1}\bigg)+(q+1)q^{a+m-1}\right]\\
		&= 2q^{-\frac{\ell(\mu)}{2}}\left[-(q^{2m-1}+1)(q^a-1)\right.\\
		&\qquad\qquad\left.+q^{a-1}(q-1)\sum_{1\leq i\leq m} (q^{2m-i}-q^i)+q^a(q^m+q^{m-1})\right]\\
		&= 2q^{-\frac{\ell(\mu)}{2}}\left[-(q^{2m-1}+1)(q^a-1)\right.\\
		&\qquad\qquad\left.+q^{a-1}(q-1)(q^m-q)\left(\frac{q^m-1}{q-1}\right)+q^a(q^m+q^{m-1})\right]\\
		&= 2q^{-\frac{\ell(\mu)}{2}}\left[-(q^{2m-1}+1)(q^a-1)
		+q^{a}(q^{m-1}-1)(q^m-1)+q^a(q^m+q^{m-1})\right]\\
		&= 2q^{-\frac{\ell(\mu)}{2}}\left[q^{2m-1}+1\right.\\
		&\qquad\qquad\left.+ q^a(-q^{2m-1}-1 +q^{2m-1}-q^m-q^{m-1}+1+q^m+q^{m-1})\right]\\
		&= 2q^{-\frac{\ell(\mu)}{2}}\left[q^{2m-1}+1\right].
\end{align*}
Next we consider the case where $\ell(\mu)=2m$ is even. As with the odd case, we may rewrite our sum as
\begin{align*}
 \orb(z_{\mu})) &=  2q^{-\frac{\ell(\mu)}{2}}\left[((-q)^{2m}-1)\frac{q-1}{q+1} (q^a-1)\frac{q+1}{q-1}\right.\\
		&\qquad\qquad+\left.\sum_{1\leq i\leq m} ((-q)^{2m-2i+1}-1)\frac{q-1}{q+1} \left((q+1)q^{a+i-1}\right)\right]\\
	&= 2q^\frac{-2m}{2}\left[(q-1)q^{a-1}\left[-\sum_{i=1}^{2m}(q^i)\right] + q^a(q^{2m}-1) +(1-(-q)^{2m})\right]\\
	&= 2q^\frac{-2m}{2}\left[(q-1)q^{a-1}(-q)\left(\frac{q^{2m}-1}{q-1}\right) + q^a(q^{2m}-1) +(1-(-q)^{2m})\right]\\
	&= 2q^\frac{-2m}{2}\left[-q^{a}(q^{2m}-1)+ q^a(q^{2m}-1) +(1-(-q)^{2m})\right]\\
	&= 2q^\frac{-2m}{2}\left[1-(-q)^{2m}\right].
\end{align*}
\end{proof}

\subsubsection{$F'$ ramified over $F$}
The structure of $T$ in this case is different from before in an important way; namely, it possesses an element $g$ with $\val_F\det(g)=1$.
\begin{lemma}\label{unramtorus}
Let $F'/F$ be ramified. Then
$$T= \left\langle\left(\begin{smallmatrix}0 & D\\ 1 & 0 \end{smallmatrix}\right)\right\rangle\times T_{\ints},$$ where $T_{\ints}$ is as above equal to $T\cap K$.
\end{lemma}
\begin{proof}
Let $g=\left(\begin{smallmatrix}0 & D\\ 1 & 0 \end{smallmatrix}\right)$, let $g'$ be an arbitrary element of $T$, and suppose that $\val_F\det(g')=m$. We will show that $g'$ may be written as $kg^m$, with $k\in T_{\ints}$. Since it is clear (from the valuations of the individual components) that $T_{\ints}\cap\left\langle g \right\rangle = \left(\begin{smallmatrix}1 & 0\\ 0 & 1 \end{smallmatrix}\right)$, we will then have the desired equality.

By assumption, $g'$ is of the form $\left(\begin{smallmatrix} \alpha & \beta D\\ \beta & \alpha\end{smallmatrix}\right)$, so that $m=\val_F(\alpha^2 - \beta^2 D)=\min(2\val_F(\alpha),2\val_F(\beta)+1)$. 
If $m$ is even, then we must have $m=2\val_F(\alpha)<2\val_F(\beta)+1$, which may be rewritten $\frac{m}{2}=\val_F(\alpha)\leq \val_F(\beta)$. It is then an easy exercise to show that $g'g^{-m}$ lies in $T_{\ints}$, and we have $g'=(g'g^{-m})g^{m}$.
If, on the other hand, $m$ is odd, then it must be that $m=2\val_F(\beta)+1<2\val_F(\alpha)$, i.e. $\frac{m-1}{2}=\val_F(\beta)< \val_F(\alpha)$. Another easy exercise gives us $g'=(g'g^{-m})g^m$, with $g'g^{-m}\in T_{\ints}$.
\end{proof}
Compare the following with corollary \ref{unramreduction}; as with that corollary, in this section the following corollary will be used in the degenerate case $E=F$; later we shall have need of it in full generality.
\begin{corollary}\label{ramreduction}
Let $F'/F$ be ramified. Then
$$\twisted(1_{IwI}) = \frac{1}{\meas_{T}(T_{\ints})}\# X_w^0(\delsig).$$
\end{corollary}
\begin{proof} 
From lemma \ref{unramtorus} we know that $T\backslash X_w(\delsig)=\left(\langle\left(\begin{smallmatrix} 0 & D\\ 1 & 0\end{smallmatrix}\right)\rangle\times T_{\ints}\right)\backslash X_w^0(\delsig)$. For every $i\in\Z$, the action of $\langle\left(\begin{smallmatrix} 0 & D\\ 1 & 0\end{smallmatrix}\right)\rangle$ on $X_w(\delsig)$ restricts to a bijection from $X_w^i(\delsig)$ to $X^0_w(\delsig)$. The result follows immediately from this and lemmas \ref{reductionone} and \ref{oneofflemma}.
\end{proof}

\begin{theorem}\label{flramorbit}
Let $F'/F$ be ramified, normalize $dg_T$ such that $\meas_{dg_T}(T_{\ints})=1$, fix a semi-simple regular elliptic $\gamma\in G$, let $\Delta(\gamma)=q^{-a-\frac{d_T+1}{2}}$, and let $\mu\in X_*(S)$. Then 
$$\orb(z_{\mu})=\begin{cases} 
			0, & \val_F\det(\gamma)\neq \size(\mu)\\
			\frac{2q^{a+1}-q-1}{q-1}, & \ell(\mu)=0\\
			q^{-\ell(\mu)/2}(1-q^{\ell(\mu)}), & \ell(\mu)\neq 0\end{cases}
$$
\end{theorem}
\begin{proof}
It is clear that the integral is zero when $\val_F\det(\gamma)\neq\size(\mu)$. When $\ell(\mu)$ is zero, lemmas \ref{berncoefs} and \ref{ramcount} imply that 
$$\orb(z_{\mu})=\frac{2q^{a+1}-q-1}{q-1}.$$

Now let $\ell(\mu)\neq 0$. If $\val_F\det(\gamma)$ is odd then so is $\ell(\mu)$; let $\ell(\mu)=2m-1$. Then lemma \ref{ramcount} (\ref{ramcountodd}) tells us that all the non-empty \xwgz\ have $w$ of even length. That same lemma, in addition to  lemma \ref{berncoefs} and corollaries \ref{berncor} and \ref{ramreduction} tell us that 
\begin{align*}
\orb(z_{\mu}) &= z_\mu^{\ell(\mu)}\#X^0_{(0,0,s)}(\gamma)+\sum_{1\leq i\leq \frac{\ell(\mu)}{2}} z_\mu^{\ell(\mu)-2i}\left[\#X^0_{(-i,0,s)}(\gamma)+\#X^0_{(i,0,s)}(\gamma)\right]\\
	&= z_\mu^{2m-1}\#X^0_{(0,0,s)}(\gamma)+\sum_{i=1}^{m-1} z_\mu^{2m-2i-1}\left[\#X^0_{(-i,0,s)}(\gamma)+\#X^0_{(i,0,s)}(\gamma)\right]\\
	&= q^{-\frac{\ell(\mu)}{2}}\left[((-q)^{2m-1}-1)\frac{q-1}{q+1}(1)
			+2\sum_{i=1}^{m-1} ((-q)^{2m-2i-1}-1)\frac{q-1}{q+1} \left(q^{i}\right)\right]\\
		&= q^{-\frac{\ell(\mu)}{2}}\frac{q-1}{q+1}\left[(q^{2m-1}+1)-2\sum_{i=0}^{m-1} (q^{2m-2i-1}+1)\left(q^{i}\right)\right]\\
		&= q^{-\frac{\ell(\mu)}{2}}\frac{q-1}{q+1}\left[(q^{2m-1}+1)-2\sum_{i=0}^{m-1} (q^{2m-i-1}+q^{i})\right].\\
\intertext{In the summation, the $q^{i}$ term runs from $q^0$ to $q^{m-1}$, while the $q^{2m-i-1}$ term runs from $q^{m}$ to $q^{2m-1}$. Adding them all together gives us}
&= q^{-\frac{\ell(\mu)}{2}}\frac{q-1}{q+1}\left[(q^{2m-1}+1)-2\left(\frac{q^{2m}-1}{q-1}\right)\right]\\
&= q^{-\frac{\ell(\mu)}{2}}\frac{1}{q+1}\left[(q^{2m}+q-q^{2m-1}-1)-2q^{2m}+2\right]\\
&= q^{-\frac{\ell(\mu)}{2}}\frac{1}{q+1}\left[-q^{2m}+q-q^{2m-1}+1\right]\\
&= q^{-\frac{\ell(\mu)}{2}}\left(1-q^{2m-1}\right).
\end{align*}

When $\val_F\det(\gamma)$ is even so is $\ell(\mu)$; let $\ell(\mu)=2m$. Lemma \ref{ramcount} (\ref{ramcounteven}) tells us that all the non-empty \xwgz\ have $w$ of odd or zero length. That lemma,  lemma \ref{berncoefs}, and corollaries \ref{berncor} and \ref{ramreduction} imply that
\begin{align*}
\orb(z_{\mu}) &= z_\mu^{\ell(\mu)}\#X^0_{(0,0,s)}(\gamma)+\sum_{1\leq i\leq \frac{\ell(\mu)+1}{2}} z_\mu^{\ell(\mu)-2i+1}\left[\#X^0_{(-i,1,s)}(\gamma)+\#X^0_{(i-1,1,s)}(\gamma)\right]\\
	&= z_\mu^{2m}\#X^0_{(0,0,s)}(\gamma)+\sum_{i=1}^m z_\mu^{2m-2i+1}\left[\#X^0_{(-i,1,s)}(\gamma)+\#X^0_{(i-1,1,s)}(\gamma)\right]\\
 	&= q^{-\frac{\ell(\mu)}{2}}\frac{q-1}{q+1}\left[((-q)^{2m}-1)\frac{2q^{a+1}-q-1}{q-1}+\sum_{i=1}^m ((-q)^{2m-2i+1}-1)2q^{a+i}\right]\\
 	&= q^{-\frac{\ell(\mu)}{2}}\frac{q-1}{q+1}\left[(q^{2m}-1)\frac{2q^{a+1}-q-1}{q-1}-2q^a\sum_{i=1}^m (q^{2m-i+1}+q^i)\right]\\
 	&= q^{-\frac{\ell(\mu)}{2}}\frac{q-1}{q+1}\left[(q^{2m}-1)\frac{2q^{a+1}-q-1}{q-1}-2q^a(q)\frac{q^{2m}-1}{q-1}\right]\\
 	&= q^{-\frac{\ell(\mu)}{2}}\frac{q^{2m}-1}{q+1}\left[2q^{a+1}-q-1-2q^{a+1}\right]\\
 	&= q^{-\frac{\ell(\mu)}{2}}(1-q^{2m}).
\end{align*}
\end{proof}




\section{Twisted Orbital Integrals}\label{twistedchapt}

The structure of this section will, for the most part, parallel that of chapter \ref{orbitalchapt}. In this chapter, $\delta\in G_E$ will be such that $N(\delta)$ is equal to a semi-simple regular elliptic $\gamma$. 

\pagebreak[4]

\subsection{Computing the twisted orbital integral}\label{twisted}

\subsubsection{$F'$ unramified over $F$}
Recall from corollary \ref{unramreduction} that
$$\twisted(1_{I_EwI_E}) = \frac{1}{\meas_{T}(K_\gamma)}\left(\# X_w^0(\delta\sigma) + \# X_{\overline{w}}^0(\delta\sigma)\right).$$ 

The following should be compared to theorem \ref{flunramorbit}.
\begin{theorem}\label{flunram}
Let $F'/F$ be unramified, normalize $d_{g_T}$ such that $\meas_{dg_T}(T_{\ints})=1$, let $\delta\in G_E$ such that $N(\delta)=\gamma$, $\gamma$ a semi-simple regular elliptic element of $G$. Let $\Delta(\gamma)=q^{-a}$ and let $z_\mu\in\zhie$, $\mu\in X_*(S_E)$.  Then
$$\twisted(z_{\mu})=\begin{cases} 
			0, & \val\det(\delta)\neq \size(\mu)\\
			2(q^a-1)\left(\frac{q+1}{q-1}\right), & \ell(\mu)=0\\
			2q^{-f\ell(\mu)/2}(1-(-q)^{f\ell(\mu)}), & \ell(\mu)\neq 0.\end{cases}
$$
\end{theorem}
\begin{proof}
We may assume without loss of generality that $\gamma$ is in standard form.
The reader may find it helpful to refer back to lemma \ref{twunramcount} and corollary \ref{berncor}. It is clear from them that $\twisted(z_{\mu})$ is zero when $\val_E\det(\delta)\neq \size(\mu)$ and that $\twisted(z_\mu)=2(q^a-1)\left(\frac{q+1}{q-1}\right)$ when $\ell(\mu)=0$.

Suppose for the rest of the proof $\ell(\mu)\neq 0$ and let $s=\val_E\det(\delta)=\size(\mu)$. Following lemma \ref{twunramcount}, there are four cases: the cases where $T$ splits over $E$ and $s$ is odd or even, and the cases where $T$ does not split over $E$, $s$ odd or even. 

\paragraph{\textbf{$T$ split over $E$, odd $s$.}}Suppose that $T$ splits over $E$, so that $f$ must be even, and $s$ is odd. The only non-empty $\xwdz$ are those with even length $w$.  Corollaries \ref{berncor} and \ref{unramreduction} tell us that
\begin{align*}
\twisted(z_\mu) &= z_{\mu}^{\ell(\mu)}\left[\#X^0_{(0,0,s)}(\delsig)+\#X^0_{(0,0,s)}(\delsig)\right]\\
		&\qquad\qquad+\sum_{1\leq i\leq \frac{\ell(\mu)-1}{2}} 
 			z_{\mu}^{\ell(\mu)-2i} \left[\left(
 			\#X^0_{(-i,0,s)}(\delsig)+\#X^0_{\overline{(-i,0,s)}}(\delsig)
 			\right)\right.\\
			&\qquad\qquad\qquad\qquad \left.+ \left(\#X_{(i,0,s)}(\delsig)+
 			\#X^0_{\overline{(i,0,s)}}(\delsig)
 			\right)\right]\\
 	&= 2\left[z_{\mu}^{\ell(\mu)}\left(\#X^0_{(0,0,s)}(\delsig)\right)\right.\\
		&\qquad\qquad+\sum_{1\leq i\leq \frac{\ell(\mu)-1}{2}} \left.
			z_{\mu}^{\ell(\mu)-2i} \left(
			\#X^0_{(-i,0,s)}(\delsig) + \#X^0_{(i,0,s)}(\delsig)\right)\right],
\end{align*}
where in the summation index we have taken into account that $\ell(\mu)$ is odd (since $\siz(\mu)$ is). If we let $\ell(\mu)=2m-1$, then we may rewrite the above as 
$$2\left[z_{\mu}^{2m-1}\left(\#X_{(0,0,s)}(\delsig)\right)
		+\sum_{i=1}^{m-1} 
			z_{\mu}^{2m-2i-1} \left(
			\#X_{(-i,0,s)}(\delsig) + \#X_{(i,0,s)}(\delsig)\right)\right].$$
Now we substitute in the values of the \xwdz\ from lemma \ref{twunramcount} (\ref{twunramcountone}) and the $z_\mu^d$ from theorem \ref{berncoefs}, yielding
\begin{align*}
\twisted(z_\mu) &= 2q^{-f\ell(\mu)/2}\left[((-q^f)^{2m-1}-1)\frac{q^f-1}{q^f+1}(1)\right.\\
			&\qquad\qquad\qquad\qquad+2\sum_{i=1}^{m-1} \left.((-q^f)^{2m-2i-1}-1)\frac{q^f-1}{q^f+1} \left(q^{fi}\right)\right]\\
		&= 2q^{-f\ell(\mu)/2}\left(\frac{q^f-1}{q^f+1}\right)\left[(q^{f(2m-1)}+1)-2\sum_{i=0}^{m-1} (q^{f(2m-2i-1)}+1)\left(q^{fi}\right)\right]\\
		&= 2q^{-f\ell(\mu)/2}\left(\frac{q^f-1}{q^f+1}\right)\left[(q^{f(2m-1)}+1)-2\sum_{i=0}^{m-1} (q^{f(2m-i-1)}+q^{fi})\right].\\
\intertext{In the summation, the $q^{fi}$ term runs from $q^0$ to $(q^f)^{m-1}$, while the $q^{f(2m-i-1)}$ term runs from $(q^f)^{m}$ to $(q^f)^{2m-1}$. Adding them all together gives us}
&= 2q^{-f\ell(\mu)/2}\left(\frac{q^f-1}{q^f+1}\right)\left[(q^{f(2m-1)}+1)-2\left(\frac{q^{f(2m)}-1}{q^{f}-1}\right)\right]\\
&= 2q^{-f\ell(\mu)/2}\left(\frac{1}{q^f+1}\right)\left[(q^{f(2m)}+q^f-q^{f(2m-1)}-1)-2q^{f(2m)}+2\right]\\
&= 2q^{-f\ell(\mu)/2}\left(\frac{1}{q^f+1}\right)\left[-q^{f(2m)}+q^f-q^{f(2m-1)}+1\right]\\
&= 2q^{-f\ell(\mu)/2}\left(1-q^{f(2m-1)}\right)\\
&= 2q^{-f\ell(\mu)/2}\left(1-(-q)^{f\ell(\mu)}\right),
\end{align*}
where the last equality follows because $f$ is even. This dispenses with the case when $T$ splits over $E$ and $s$ is odd.

There are two simplifications common to all three of our remaining cases. First, in each of them the only non-empty $\xwdz$ are those with $w$ of odd or zero length, following lemma \ref{twunramcount} (\ref{twunramcounttwo}). Second, in each case the sum over the odd length terms may be simplified as 
\begin{align}
\left(\#X^0_{(-i,1,s)}(\delsig)\right.&+\left.\#X^0_{\overline{(-i,1,s)}}(\delsig)\right)+ \left(\#X^0_{(i-1,1,s)}(\delsig)+\#X^0_{\overline{(i-1,1,s)}}(\delsig)\right)\notag\\
 &= 2\left(\#X^0_{(-i,1,s)}(\delsig)+\#X^0_{(i-1,1,s)}(\delsig)\right)\notag\\
&= 2q^{fi}\left[(1-q^{1-f})\left(\frac{q^{a-1}-1}{q-1}+\frac{q^a-1}{q-1}\right)+(q+1)q^{a-1}\right]\notag\\
&= 2q^{fi}\left[(1-q^{1-f})\frac{q^a+q^{a-1}-2}{q-1}+(q+1)q^{a-1}\right]\label{longshort}.
\end{align}
The last equality follows because $a$ is fixed, and, for any given value of $i$, $a+i$ is either even or odd, and in either case we add one long term  and one short term from \ref{twunramcount} (\ref{twunramcounttwo}). 

\paragraph{\textbf{$T$ split over $E$, even $s$.}}Now suppose that $T$ splits over $E$, so that $f$ is again even, and $s$ is even, so that $\ell(\mu)$ is as well. Set $\ell(\mu)=2m$. Then (\ref{longshort}), theorem \ref{berncoefs} and corollary \ref{berncor} show that
\begin{align*}
\twisted(z_\mu) &= 2z_\mu^{2m}\#X^0_{(0,0,s)}+ 2\sum_{1\leq i\leq m}z_\mu^{2m-2i+1}\left[\#X^0_{(-i,1,s)}+\#X^0_{(i-1,1,s)}\right]\\
		&= 2q^{-f\frac{\ell(\mu)}{2}} \left[((-q^f)^{2m}-1)\frac{q^f-1}{q^f+1}\frac{q+1}{q-1}(q^a-1)\right.\\
		&\left.\qquad\qquad+ \sum_{i=1}^m ((-q^f)^{2m-2i+1}-1)\frac{q^f-1}{q^f+1}q^{fi}\right.\\
		&\qquad\qquad\qquad\cdot\left.\left((1-q^{1-f})\frac{q^a+q^{a-1}-2}{q-1}+(q+1)q^{a-1}\right)\right]\\
		&= \left.2q^{-f\frac{\ell(\mu)}{2}} \frac{q^f-1}{(q^f+1)(q-1)}\right[(q^{f(2m)}-1)(q+1)(q^a-1)\\
		&\left.\qquad\qquad\qquad- \left((1-q^{1-f})(q^a+q^{a-1}-2)+(q^2-1)q^{a-1}\right)\right.\\
		&\qquad\qquad\qquad \cdot\left.\sum_{i=1}^m(q^{f(2m-i+1)}+q^{fi})\right].\\
\intertext{Simplifying the summation gives}
		&= \left.2q^{-f\frac{\ell(\mu)}{2}} \frac{q^f-1}{(q^f+1)(q-1)}\right[(q^{f(2m)}-1)(q+1)(q^a-1)\\
		&\left.\qquad\qquad\qquad- \left((1-q^{1-f})(q^a+q^{a-1}-2)+(q^2-1)q^{a-1}\right)q^f\frac{q^{f(2m)}-1}{q^f-1}\right]\\
		&= 2q^{-f\frac{\ell(\mu)}{2}} \frac{q^{f(2m)}-1}{(q^f+1)(q-1)}\left[(q^{f}-1)(q+1)(q^a-1)\right.\\
		&\qquad\qquad\qquad \left.- \left((q^{f}-q)(q^{a-1}(q+1)-2)+q^f(q^2-1)q^{a-1}\right)\right]\\
		&= 2q^{-f\frac{\ell(\mu)}{2}} \frac{q^{f(2m)}-1}{(q^f+1)(q-1)}\left[q^{a-1}(q+1)\bigg(q(q^f-1)-(q^f-q)-q^f(q-1)\bigg)\right.\\
		&\qquad\qquad\qquad\qquad\qquad\qquad\left.-(q^f-1)(q+1)+2(q^f-q)\right]\\
		&= 2q^{-f\frac{\ell(\mu)}{2}} \frac{q^{f(2m)}-1}{(q^f+1)(q-1)}\left[-q^{f+1}+q^f-q+1\right]\\
		&= 2q^{-f\frac{\ell(\mu)}{2}} (1-q^{f(2m)})\\
		&= 2q^{-f\frac{\ell(\mu)}{2}} (1-(-q)^{f\ell(\mu)}).
\end{align*}

\paragraph{\textbf{$T$ not split over $E$, odd $s$.}}Suppose that $T$ does not split over $E$, so that $f$ is odd, and suppose $s$ is odd. Then $\ell(\mu)$ is as well, and we may write $\ell(\mu)=2m-1$. In this case, using (\ref{longshort}) and theorem \ref{berncoefs} yields 
\begin{align*}
\twisted(z_\mu) &= 2z_\mu^{2m-1}\frac{q+1}{q-1}(q^a-1)\\
		&\qquad+ 2\sum_{1\leq i\leq m}
		z_\mu^{2m-2i}\left[q^{fi}\left((1-q^{1-f})\frac{q^a+q^{a-1}-2}{q-1}+(q+1)q^{a-1}\right)\right]\\
		&= 2q^{-f\frac{\ell(\mu)}{2}} \left[((-q^f)^{2m-1}-1)\frac{q^f-1}{q^f+1}\frac{q+1}{q-1}(q^a-1)\right.\\
		&\qquad+\left((1-q^{1-f})\frac{q^a+q^{a-1}-2}{q-1}+(q+1)q^{a-1}\right)\\
		&\left.\qquad\qquad\qquad\qquad\qquad\cdot \left(q^{f(m)}+\sum_{i=1}^{m-1} ((-q^f)^{2m-2i}-1)\frac{q^f-1}{q^f+1}q^{fi}\right)\right]\\
		&= 2q^{-f\frac{\ell(\mu)}{2}}\frac{q^f-1}{(q^f+1)(q-1)}
		\left[-(q^{f(2m-1)}+1)(q+1)(q^a-1)\right.\\
		&\left.\qquad+\left((1-q^{1-f})(q^a+q^{a-1}-2)+(q^2-1)q^{a-1}\right)\right.\\
		&\qquad\qquad\qquad\qquad\qquad \left.\cdot\left(q^{f(m)}\frac{q^f+1}{q^f-1}+\sum_{i=1}^{m-1} (q^{f(2m-i)}-q^{fi})\right)\right]\\
\intertext{Within the summation, the $q^{f(2m-i)}$ terms run from $q^{f(m+1)}$ to $q^{f(2m-1)}$, while the $q^{fi}$ terms run from $q^f$ to $q^{f(m-1)}$. Factoring $q^{f(m+1)}$ from the first terms and $q^f$ from the second yields}
		&= 2q^{-f\frac{\ell(\mu)}{2}}\frac{q^f-1}{(q^f+1)(q-1)} \left[-(q^{f(2m-1)}+1)(q+1)(q^a-1)\right.\\
		&\left.\qquad+\left((1-q^{1-f})(q^a+q^{a-1}-2)+(q^2-1)q^{a-1}\right)\right.\\
		&\qquad\qquad\qquad\qquad\left. \cdot\left(q^{f(m)}\frac{q^f+1}{q^f-1}+(q^{f(m+1)}-q^f)\frac{q^{f(m-1)}-1}{q^f-1}\right)\right]\\
		&= 2q^{-f\frac{\ell(\mu)}{2}}\frac{1}{(q^f+1)(q-1)} \left[-(q^{f(2m-1)}+1)(q+1)(q^a-1)(q^f-1)\right.\\
		&\left.\qquad+\left((1-q^{1-f})(q^a+q^{a-1}-2)+(q^2-1)q^{a-1}\right)\right.\\
		&\qquad\qquad\qquad\qquad\left.\cdot\left(q^{f(m)}(q^f+1)+(q^{f(m+1)}-q^f)(q^{f(m-1)}-1)\right)\right]\\
		&= 2q^{-f\frac{\ell(\mu)}{2}}\frac{1}{(q^f+1)(q-1)} \left[-(q^{f(2m-1)}+1)(q+1)(q^a-1)(q^f-1)\right.\\
		&\left.\qquad+\left((1-q^{1-f})(q^a+q^{a-1}-2)+(q^2-1)q^{a-1}\right) q^f\left(q^{f(2m-1)}+1\right)\right]\\
		&= 2q^{-f\frac{\ell(\mu)}{2}}\frac{q^{f(2m-1)}+1}{(q^f+1)(q-1)} \left[-(q+1)(q^a-1)(q^f-1)\right.\\
		&\qquad\qquad\qquad\qquad\qquad\qquad\left.+\left((q^f-q)(q^a+q^{a-1}-2)+(q^2-1)q^fq^{a-1}\right) \right]\\
		&= 2q^{-f\frac{\ell(\mu)}{2}}\frac{q^{f(2m-1)}+1}{(q^f+1)(q-1)} \left[q^{a-1}(q+1)\left((q^f-q)+(q-1)q^f-q(q^f-1) \right)\right.\\ &\left.\qquad\qquad\qquad\qquad\qquad\qquad+\left((q+1)(q^f-1)-2(q^f-q)\right)\right]\\
		&= 2q^{-f\frac{\ell(\mu)}{2}}\frac{q^{f(2m-1)}+1}{(q^f+1)(q-1)}\left[q^{f+1}-q^f+q-1\right]\\	 
		&= 2q^{-f\frac{\ell(\mu)}{2}}(1-(-q)^{f\ell(\mu)}),
\end{align*}
since both $f$ and $\ell(\mu)$ are odd.

\paragraph{\textbf{$T$ not split in $E$, even $s$.}} Finally, consider the case $T$ not split in $E$, so that $f$ is again odd, and $s$ (and therefore $\ell(\mu)$) is even. Set $\ell(\mu)=2m$. Just as above, (\ref{longshort}) and theorem \ref{berncoefs} gives us
\begin{align*}
\twisted(z_\mu) &= 2z_\mu^{2m}\#X^0_{(0,0,s)}(\delsig)+ 2\sum_{1\leq i\leq m}z_\mu^{2m-2i+1}\left[\#X^0_{(-i,1,s)}(\delsig)+\#X^0_{(i-1,1,s)}(\delsig)\right]\\
		&= 2q^{-f\frac{\ell(\mu)}{2}} \left[((-q^f)^{2m}-1)\frac{q^f-1}{q^f+1}\frac{q+1}{q-1}(q^a-1)\right.\\
		&\left.\qquad\qquad\qquad+ \sum_{i=1}^m ((-q^f)^{2m-2i+1}-1)\frac{q^f-1}{q^f+1}q^{fi}\right.\\
		&\qquad\qquad\qquad\left.\cdot\left((1-q^{1-f})\frac{q^a+q^{a-1}-2}{q-1}+(q+1)q^{a-1}\right)\right]
\end{align*}
Comparing this with the sum for $T$ splitting in $E$, $\ell(\mu)$ even shows that the current sum is exactly the same; for even though $f$is odd in this case it occurs in the sum as $-(q^f)$, and it does not determine any signs.
\end{proof}

\subsubsection{$F'$ ramified over $F$}
Recall from corollary \ref{ramreduction} that
$$\twisted(1_{IwI}) = \frac{1}{\meas_{T}(T_{\ints})}\# X_w^0(\delsig).$$
The following should be compared with theorem \ref{flramorbit}.
\begin{theorem}\label{flram}
Let $F'/F$ be ramified, normalize $dg_T$ so that $\meas_{dg_T}(T_{\ints})=1$, fix a semi-simple regular elliptic $\gamma\in G$, let $d_T$ be the distance from the barycenter of $e_T$ to $\apt(EF')$, let $\Delta(\gamma)=q^{-a+\-\frac{d_T+1}{2}}$, let $\mu\in X_*(S_E)$, and let $z_\mu\in \zhie$. Then 
$$\twisted(z_{\mu}))=\begin{cases}
			0, & \val\det(\delta)\neq \size(\mu)\\
			\frac{2q^{a+1}-q-1}{q-1}, & \ell(\mu)=0\\
			q^{-f\ell(\mu)/2}(1-q^{f\ell(\mu)}), & \ell(\mu)\neq 0.\end{cases}$$
\end{theorem}
\begin{proof}
The integral is clearly zero if $\val_E\det(\delta)\neq\size(\mu)$. Suppose for the remainder of the proof that $s=\val_E\det(\delta)=\siz(\mu)$. When $\ell(\mu)=0$, lemmas \ref{berncoefs} and \ref{twramcount} (\ref{twramcounttwo}) immediately give us the result in the case $\ell(\mu)=0$, so suppose $\ell(\mu)\neq 0$. 

\paragraph{\textbf{Odd $s$.}}Suppose $s$ is odd, and let $\ell(\mu)=2m-1$. 
Lemma \ref{twramcount} (\ref{twramcountone}) tells us that the only non-empty \xwdz\ are those with $w$ of even length. Then lemma \ref{berncoefs} and corollary \ref{berncor} imply that 
\begin{align*}
\twisted(z_\mu) &= z_\mu^{\ell(\mu)}\#X^0_{(0,0,s)}(\delsig)+\sum_{1\leq j\leq \frac{\ell(\mu)}{2}} z_\mu^{\ell(\mu)-2j}\left[\#X^0_{(-j,0,s)}(\delsig)+\#X^0_{(j,0,s)}(\delsig)\right]\\
	&= z_\mu^{2m-1}\#X^0_{(0,0,s)}(\delsig)+\sum_{i=1}^{m-1} z_\mu^{2m-2j-1}\left[\#X^0_{(-j,0,s)}(\delsig)+\#X^0_{(j,0,s)}(\delsig)\right]
\end{align*}
This is exactly one half the formula found at the beginning of the proof of theorem \ref{flunram}; indeed, the sum is over the same index, and all of the coefficients $z_m^i$ and $\#X^0_w(\delsig)$ are the same. We may follow the reasoning from that case to its penultimate conclusion, that the twisted orbital integral is equal to $q^{-f\ell(\mu)/2}\left(1-q^{f(2m-1)}\right)$, as desired. The final line of the proof of \ref{flunram} may or may not hold in this case, as we have not made any assumptions about the parity of $f$ in this proof.

\paragraph{\textbf{Even $s$.}} Finally, suppose that $s$ is even, and let $\ell(\mu)=2m$. Lemma \ref{twramcount} (\ref{twramcountone}) tells us that the only non-empty \xwdz\ are those with $w$ of odd or zero length. Then lemma \ref{berncoefs} and corollary \ref{berncor} yield
\begin{align*}
\twisted(z_\mu) &= z_\mu^{\ell(\mu)}\#X^0_{(0,0,s)}(\delsig)+\sum_{1\leq i\leq \frac{\ell(\mu)+1}{2}} z_\mu^{\ell(\mu)-2i+1}\left[\#X^0_{(-i,1,s)}(\delsig)+\#X^0_{(i-1,1,s)}(\delsig)\right]\\
	&= z_\mu^{2m}\#X^0_{(0,0,s)}(\delsig)+\sum_{i=1}^m z_\mu^{2m-2i+1}\left[\#X^0_{(-i,1,s)}(\delsig)+\#X^0_{(i-1,1,s)}(\delsig)\right]\\
	&= q^{-f\frac{\ell(\mu)}{2}}\left[((-q^f)^{2m}-1)\left(\frac{q^f-1}{q^f+1}\right)\left(\frac{2q^{a+1}-q-1}{q-1}\right)\right.\\ &\qquad\qquad+\left.2\sum_{i=1}^m ((-q^f)^{2m-2i+1}-1)\frac{q^f-1}{q^f+1}q^{fi}\left(q^a+\frac{q^{a}-1}{q-1}(1-q^{1-f})\right)\right]\\
	&= q^{-f\frac{\ell(\mu)}{2}}\frac{q^f-1}{(q^f+1)(q-1)}\left[(q^{f(2m)}-1)(2q^{a+1}-q-1)\right.\\ &\qquad\qquad\qquad -\left.2\left(q^a(q-1)+(q^{a}-1)(1-q^{1-f})\right)\sum_{i=1}^m (q^{f(2m-i+1)}+q^{fi})\right]\\
	&= q^{-f\frac{\ell(\mu)}{2}}\frac{q^f-1}{(q^f+1)(q-1)}\left[(q^{f(2m)}-1)(2q^{a+1}-q-1)\right.\\ &\qquad\qquad\qquad -\left.2\left(q^a(q-1)+(q^{a}-1)(1-q^{1-f})\right)q^f\frac{q^{f(2m)}-1}{q^f-1}\right]\\
	&= q^{-f\frac{\ell(\mu)}{2}}\frac{q^{f(2m)}-1}{(q^f+1)(q-1)}\left[(q^f-1)(2q^{a+1}-q-1)\right.\\
	&\qquad\qquad\left.-2\left(q^aq^f(q-1)+(q^{a}-1)(q^f-q)\right)\right]\\
	&= q^{-f\frac{\ell(\mu)}{2}}\frac{q^{f(2m)}-1}{(q^f+1)(q-1)}\left[-(q^f-1)(q+1)\right.\\
	&\qquad\qquad\left.+2(q^f-q) +2q^a\left(q(q^f-1) -q^f(q-1)-(q^f-q)\right)\right]\\
	&= q^{-f\frac{\ell(\mu)}{2}}\frac{q^{f(2m)}-1}{(q^f+1)(q-1)}\left[-q^{f+1}-q^f+q+1+2q^f-2q\right.\\
	&\qquad\qquad\left.+2q^a\left(q^{f+1}-q -q^{f+1}+q^f-q^f+q)\right)\right]\\
	&= q^{-f\frac{\ell(\mu)}{2}}\frac{q^{f(2m)}-1}{(q^f+1)(q-1)}\left[-q^{f+1}+q^f-q+1\right]\\
	&= q^{-f\frac{\ell(\mu)}{2}}(1-q^{f(2m)}).
\end{align*} 
\end{proof}



\section{The Base Change Lemma}\label{fltchapt}

\subsection{Statement of the Theorem}
\begin{theorem}\label{funlemma}
Let $\gamma\in G$ be semi-simple and regular, and let $\phi\in\zhie$. Then 
$$\orb(b\phi) = \begin{cases} 0, & \gamma\text{ not a norm}\\
	\twisted(\phi), & \gamma\text{ is a norm of }\delta.
\end{cases}
$$
\end{theorem}
Theorem \ref{berncoefs} implies that we need only prove the lemma for $\phi=z_\mu$ for all dominant $\mu\in X_*(S_E)$.

\subsection{Reduction to the Case of Semi-Simple Regular Elliptic $\gamma$}
\begin{lemma}\label{ellipticonly}
Theorem \ref{funlemma} is true for $\gamma$ that are semi-simple regular but not elliptic.
\end{lemma}
Proving the lemma reduces the proof of theorem \ref{funlemma} to the case where $\gamma$ is semi-simple regular elliptic.
\begin{proof}
Then work in \cite[subsections 4.2 and 4.4]{haineslemma} imply this result; we here briefly review the arguments, using the notation and terminology of \cite{haineslemma}. 

Recall that there are only two kinds or semi-simple regular elements, those which lie in an elliptic torus and those which lie in an  $F$-split torus. For $\gamma$ that lie in a torus $S'$ split over $F$,
lemma \cite[4.4.2]{haineslemma} tells us that either $\gamma$ is not a norm or there exists a $\delta\in S'_E$ such that $\gamma=N(\delta)$. The results \cite[4.4.4 and 4.4.6]{haineslemma} show that theorem \ref{funlemma}  holds for $\gamma\in S'$  that are the norm of some $\delta\in S'_E$ and for $\gamma\in S'$ that are not norms if and only if the analog of theorem \ref{funlemma} with $G$ replaced by $S'$ is true. This analog is obvious, completing the proof.
\end{proof}

\subsection{Semi-Simple Regular Elliptic $\gamma$ Not a Norm}
\begin{lemma}Let $\gamma\in G$ be semi-simple regular elliptic. If $\gamma$ is not a norm of $\delta$ for any $\delta\in G_E$ then $\orb(b(z_{E\mu}))=0$ for any $\mu$.
\end{lemma}
\begin{proof}
The only characteristic functions $1_{IwI}$ appearing with non-zero coefficients in $b(z_{\mu})=z_{f\cdot\mu}$ have $w$ that are $f\cdot\mu$-admissible, and so have size divisible by $f$. Theorem \ref{distance} dictates that the size of all $w$ such that $g^{-1}\gamma g\in IwI$ for some $g$ is $\val_F\det(\gamma)$. Since $\gamma\in N(GL_2(E))$ if and only if $f$ divides $\val_F\det(\gamma)$ \cite[4.7]{langlands}, $f$ must not divide $\val_F\det(\gamma)$, nor the size of the $w$ that have $\{g^{-1}\gamma g:g\in G\}\cap IwI \neq \emptyset$.
\end{proof}

\subsection{Semi-Simple Regular Elliptic $\gamma=N(\delta)$}
Lemma \ref{associated} (\ref{assocone}), and the invariance of $\twisted$ under $\sigma$-conjugation of $\delta$ and of $\orb$ under conjugation of $\gamma$ allow us to assume without loss of generality that $N(\delta)=\gamma$. 
Recall from theorem \ref{basechangeforbernstein} that $b(z_\mu)=z_{f\mu}$. Then the lemma follows from comparing the values of $\twisted(z_\mu)$ from section \ref{twistedchapt} to the values of $\orb(z_{f\mu})$ from chapter \ref{orbitalchapt}, paying close attention to the ramification (or not) of $F'/F$. \begin{flushright}\hfill$\qed$\end{flushright}

%
%
%
%
%
%
%
%
\section{Acknowledgements}
I would like to thank my advisor, Prof.\ Thomas Haines, for suggesting this line of research, for many fruitful conversations about the background of this problem and comments regarding the style and substance of this work, for his advice throughout the last 4 (going on 5) years, for his support mathematically and fiscally, and, most of all, I think, for his patience.
 
I would also like to thank Moshe Adrian for discussing with me the relevant aspects of elliptic tori in $GL_2$, and Kevin Wilson for assistance with details about the structure of the Bruhat-Tits buildings of $GL_2$ and $SL_2$, for listening to my ideas, and for being willing and able to point out their (many) flaws. 

Many friends have helped in ways mathematical, logistical, and social in the years leading up to this paper. In addition to those already named, Randy Baden, Ted Clifford, Ben and Laura Lauser, Beth McLaughlin, Chris and Kate Truman, you have my gratitude and my thanks.

Some smaller friends have also provided absolutely necessary distractions. Izzy, Eve, thanks.

My extended family, Kak and Barb, have helped out on numerous occasions, not least helping keep things in order after each birth. Thanks, guys.

Finally, I would like to thank my wife Naomi and two children Connor and Samuel. Thanks for putting up with me. You fill my life with joy, fun, and purpose. This is for you.

\bibliographystyle{amsplain}
\bibliography{thesis} 

\end{document}